\documentclass[11pt]{amsart}
\usepackage{amscd,amsmath,amssymb,amsfonts,verbatim}
\usepackage[cmtip, all]{xy}
\usepackage{MnSymbol}

\usepackage{comment}
\usepackage{tikz-cd}

\usepackage[ left=3cm, right=3cm, top = 2.8cm, bottom = 2.8cm ]{geometry}
\newtheorem{thm}{Theorem}[section]
\newtheorem{prop}[thm]{Proposition}
\newtheorem{lem}[thm]{Lemma}
\newtheorem{cor}[thm]{Corollary}

\theoremstyle{definition}
\newtheorem{ques}[thm]{Question}

\newtheorem{defn}[thm]{Definition}
\theoremstyle{remark}
\newtheorem{remk}[thm]{Remark}

\numberwithin{equation}{section}

\newcommand{\sH}{{\mathcal H}}
\newcommand{\sI}{{\mathcal I}}

\newcommand{\sK}{{\mathcal K}}

\newcommand{\sO}{{\mathcal O}}

\newcommand{\sR}{{\mathcal R}}
\newcommand{\sS}{{\mathcal S}}

\newcommand{\sZ}{{\mathcal Z}}

\newcommand{\G}{{\mathbb G}}

\renewcommand{\P}{{\mathbb P}}
\newcommand{\Q}{{\mathbb Q}}

\newcommand{\Z}{{\mathbb Z}}

\newcommand{\Ker}{{\rm Ker}}

\newcommand{\Alb}{{\rm Alb}}
\newcommand{\CH}{{\rm CH}}

\newcommand{\surj}{\twoheadrightarrow}
\newcommand{\inj}{\hookrightarrow}
\newcommand{\red}{{\rm red}}

\newcommand{\Pic}{{\rm Pic}}

\newcommand{\Hom}{{\rm Hom}}

\newcommand{\Spec}{{\rm Spec \,}}
\newcommand{\sing}{{\rm sing}}

\newcommand{\ab}{\rm ab}

\newcommand{\divf}{{\rm div}}

\newcommand{\Sch}{{\operatorname{\mathbf{Sch}}}}

\newcommand{\Sm}{{\mathbf{Sm}}}

\newcommand{\hocolim}{\mathop{{\rm hocolim}}}

\newcommand{\cyc}{{\operatorname{\rm cyc}}}

\newcommand{\ds}{{/\kern-3pt/}}

\newcommand{\lci}{{\rm l.c.i.\!}}

\newcommand{\ov}{\overline}

\renewcommand{\dim}{\text{\rm dim}}

\newcommand{\tuborg}{\left\{\begin{array}{ll}}
\newcommand{\sluttuborg}{\end{array}\right.}

\newcommand{\zar}{{\rm zar}}

\newcommand{\reg}{{\rm reg}}

\newcommand{\tor}{{\rm tor}}

\newcommand{\abk}{{\rm abk}}

\newcommand{\adiv}{{\rm adiv}}

\newcommand{\wt}{\widetilde}

\newcommand{\Rus}{{\rm Rus}}
\newcommand{\cdh}{{\rm cdh}}
\newcommand{\Char}{{\rm char}}
\newcommand{\alb}{{\rm alb}}

\def\cO{\mathcal{O}}

\newcounter{elno}

\newcounter{elno-abc}   

\newcounter{elno-abc-prime}

\begin{document}
\title{Suslin homology via cycles with modulus and applications}
\author{Federico Binda, Amalendu Krishna}
%\address
\address{Dipartimento di Matematica ``Federigo Enriques'',
  Universit\`a degli Studi di Milano\\ Via Cesare Saldini 50, 20133 Milano,
  Italy}
\email{federico.binda@unimi.it}
\address{Department of Mathematics, Indian Institute of Science,  
  Bangalore, 560012, India.}
%\email
\email{amalenduk@iisc.ac.in}
%\address{School of Mathematics, Tata Institute of Fundamental Research, Homi Bhabha
%Road, Mumbai-400005, India.}

\keywords{$K$-theory, Algebraic cycles, Motivic cohomology}        

\subjclass[2020]{Primary 14C25; Secondary 14F42, 19E15}

\maketitle

\begin{quote}\emph{Abstract.}
  We show that for a smooth projective variety $X$ over a field $k$ and a reduced effective
  Cartier divisor $D \subset X$, the Chow group of 0-cycles with modulus $\CH_0(X|D)$
  coincides with the Suslin homology $H^S_0(X \setminus D)$ under some necessary conditions
  on $k$ and $D$. We derive several consequences, and  we answer to a question of Barbieri-Viale and Kahn.
%This implies thatc$\CH_0(X|D)$ also coincides with Russell's relative Chow group in these cases. The proof uses the Levine-Weibel Chow groups of singular varieties.
\end{quote}
%\end{abstract}
\setcounter{tocdepth}{1}
%\maketitle
\tableofcontents

\section{Introduction}\label{sec:Intro}
The theory of Chow groups with modulus is presently an active area of research whose
primary goal is to provide a cycle theoretic description of the relative $K$-theory of
smooth varieties and of the (ordinary) $K$-theory of singular varieties.
As such, this is a non-$\mathbb{A}^1$-homotopy invariant cohomology theory: this  poses
a major hurdle while dealing with Chow groups with modulus. The question then
arises whether one can isolate a number of special cases in which  the Chow groups with modulus
behave like a homotopy invariant cohomology theory. This note is an attempt to answer
this question.

More specifically, we exhibit a phenomenon which justifies the belief that
the Chow groups with modulus associated to a normal crossing divisor on a smooth
scheme over a field should behave like a homotopy invariant theory.
The precise result that we prove is the following.

\subsection{Main result}\label{sec:MR}
Let $k$ be a field and $X$ a smooth projective scheme of pure dimension $d \ge 0$
over $k$. Let $D \subset X$ be a reduced effective Cartier divisor (possibly empty)
on $X$ with complement
$U$. In this case, we shall say that $(X,D)$ is a reduced modulus pair.
Let $\CH_0(X|D)$ be the Chow group of 0-cycles on $X$ with modulus $D$
(see \cite{Kerz-Saito-Duke}). Let $H^S_0(U)$ denote the (zeroth)
  Suslin homology of $U$ (see \cite[Defn.~10.8]{MVW}, where
  it is called the algebraic singular homology).
  If $k$ admits resolution of singularities, then $H^S_0(U)$ coincides with the 
  Suslin-Voevodsky motivic cohomology with compact support $H^{2d}_c(U, \Z(d))$.
  There is a canonical surjection (e.g., using \cite[Thm.~5.1]{Schmidt-ant})
  \begin{equation}\label{eqn:Surjection}
    \phi_{X|D} \colon \CH_0(X|D) \surj H^S_0(U).
  \end{equation}

  This map is clearly an isomorphism if $d \le  1$.
  However, it is known (see \cite[Thm.~1.1]{Binda-Krishna-21}) that $\phi_{X|D}$ may not
  be an isomorphism if $d \ge 2$ (even if $k$ is algebraically closed).
  The goal of this paper is to prove the following result.

\begin{thm}\label{thm:Main}
  Assume that one of the following conditions holds.
  \begin{enumerate}
  \item
  $D$ is a simple normal crossing divisor on $X$.
  \item
    $k$ is perfect, $d \le 2$ and $D$ is seminormal.
  \item
    $k$ is algebraically closed of positive characteristic.
  \item
    $k \subseteq \ov{\Q}$.
  \end{enumerate}
Then the map
    \[
      \phi_{X|D} \colon \CH_0(X|D) \to H^S_0(U)
    \]
    is an isomorphism.
  \end{thm}

  \vskip .3cm

  The case (3) of Theorem~\ref{thm:Main}
  is known by \cite[Thm.~1.1]{Binda-Krishna-21} (we included it in the theorem only for
  completeness). Hence, the new results are (1), (2) and (4).
  We expect the condition $d \le 2$ in (2) to be unnecessary, but do not know
  how to show it. However, the condition on $\Char(k)$ in (3) and algebraicity of
  $k$ over $\Q$ in (4) can not be relaxed (see  \cite[Thm.~4.4]{Binda-Krishna-21}).
%The case (1) of Theorem \ref{thm:Main} is used in \cite{Binda-Krishna-21} to deduce several comparison results between different groups of zero cycles.
We also remark that $\phi_{X|D}$ is almost never an isomorphism if $D$ is not reduced.

\subsection{Applications}\label{sec:App}
As Theorem \ref{thm:Main} identifies an a priori non-homotopy invariant theory with a
homotopy invariant one, we expect it to have many consequences. We list some in this
paper.

\subsubsection{Class field theory of Kerz-Saito}\label{sec:KSCFT}
The goal of geometric class field theory is to describe the abelian fundamental group of a variety (say, defined over a finite field) in terms of certain groups of algebraic cycles. The modern perspective on the problem is given by the work of Kerz and Saito \cite{Kerz-Saito-Duke} (see also \cite{BKS}), where the class groups used to describe the abelian fundamental group of a variety $X$ with bounded ramification along a divisor $D$ is precisely the Chow group of zero-cycles with modulus.

By a clever induction argument on the ramification index,  the proof of the
  main theorem of \cite{Kerz-Saito-Duke}   uses as key ingredient the existence of a  reciprocity isomorphism
  \begin{equation}\label{eqn:KS}
    \rho_{X|D} \colon \CH_0(X|D)_0 \xrightarrow{\cong} \pi^{\ab}_1(X,D)_0
  \end{equation}
for a reduced simple normal crossing divisor $D$ 
  on a smooth projective surface $X$ over a finite field, where $\pi^{\ab}_1(X,D)_0$ is (the degree zero part of) the abelian fundamental group of $X$ with modulus $D$, a quotient of the usual \'etale fundamental group $\pi_1^{\ab}(U)$, where $U = X\setminus D$. For the proof of this fact, Kerz and Saito refer to a result of Kerz-Schmidt
  (see \cite[Thm.~8.3]{Kerz-Schmidt}) that, reformulated in an appropriate way, affirms the existence of an isomorphism
  \begin{equation}\label{eqn:KS-0}
   \rho^t_U \colon H^S_0(U)_0 \xrightarrow{\cong} \pi^{t, \ab}_1(U)_0, 
 \end{equation}
 where  $\pi^{t, \ab}_1(U)$ is the tame
 fundamental group of $U$ (classifying tame finite \'etale coverings of $U$), a further quotient of $\pi_1^{\ab}(U)$. The comparison between \eqref{eqn:KS} and \eqref{eqn:KS-0} is very indirect, and passes through non-trivial results in ramification theory. 

An immediate application of Theorem \ref{thm:Main} is that the isomorphism
~\eqref{eqn:KS} is in fact a direct corollary of ~\eqref{eqn:KS-0}, and holds
in any dimension.
 It also follows immediately from Theorem \ref{thm:Main} and Kerz-Schmidt theorem that
 for any smooth projective variety $X$ over a finite field and a
 reduced simple normal crossing divisor $D \subset X$ with complement $U$, the canonical
 map
 \[
   \pi^{\ab}_1(X,D) \to  \pi^{t, \ab}_1(U)
 \]
 is an isomorphism of topological groups.

 \subsubsection{Reciprocity for Russell's relative Chow group}\label{sec:RCG}
An independent theory of relative Chow groups with modulus was introduced and extensively
studied
  by Russell (see \cite{Russell-ant}). For $D \subset X$ an effective
  Cartier divisor, let us denote
  Russell's relative Chow group of 0-cycles by $\CH^{\rm Rus}_0(X|D)$.
  It is clear from the definition of this group (see op. cit.)
  and \cite[Thm.~5.1]{Schmidt-ant}
  that it coincides with $H^S_0(U)$ when $D$ is reduced.  We thus immediately get the
  following.

  \begin{cor}\label{cor:Russell}
  Under the hypotheses of Theorem \ref{thm:Main}, the canonical map
    \begin{equation}\label{cor:Main-0}
      \phi_{X|D} \colon \CH_0(X|D) \to \CH^{\Rus}_0(X|D)
    \end{equation}
is an isomorphism.
\end{cor}

Combining \eqref{cor:Main-0} with the main results of \cite{BKS},
  \cite{Kerz-Saito-Duke} and \cite[Thm.~1.4]{Gupta-Krishna-BF},
  we obtain the following reciprocity theorem for
  Russell's relative Chow group. 
  Let $\pi^{\abk}_1(X,D)$ be the log version
  (see \cite[Defn.~9.6]{Gupta-Krishna-Duality} or \cite[Defn.~7.2]{Barrientos}) 
  of the non-log abelian fundamental group with modulus $\pi^{\adiv}_1(X,D)$
  (see \cite[Defn.~7.5]{Gupta-Krishna-REC}). The latter group coincides with
  the fundamental group with modulus $\pi^{\ab}_1(X,D)$
  used in \cite{BKS} and \cite{Kerz-Saito-Duke}.

\begin{thm}\label{thm:Main-0-0}
    Assume that $k$ is finite and $(X,D)$ is a reduced modulus pair over $k$ such that
    one of the following conditions holds.
    \begin{enumerate}
    \item
     $D$ is a simple normal crossing divisor.
    \item
      $d \le 2$ and $D$ is seminormal.
      \end{enumerate}
Then the Frobenius substitution at the
    closed points of $X \setminus D$ gives rise to a reciprocity isomorphism
    \[
      \rho_{X|D} \colon \CH^{\Rus}_0(X|D)_0 \xrightarrow{\cong}
      \pi^{\abk}_1(X,D)_0
    \]
    of finite groups.
    \end{thm}

    When $\Char(k) \neq 2$ in case (1),
    the theorem was claimed by Barrientos (see \cite[Thm.~ 7.3]{Barrientos}).
    For the proof,  Barrientos only refers to the (highly intricate) arguments
    of Kerz-Saito \cite{Kerz-Saito-Duke} in the non-log case.
    Note that the canonical map $\pi^{\adiv}_1(X,D) \to \pi^{\abk}_1(X,D)$
    is an isomorphism under the assumption of the corollary. This
    follows directly from definitions.

\subsubsection{Roitman's theorem for Suslin homology}\label{sec:RSH}
Assume that $k$ is algebraically closed. Let $(X,D)$ be a reduced modulus pair
over $k$ with $U = X \setminus D$. Let $\Alb(U)$ denote the generalized
Albanese variety of $U$, introduced by Serre \cite{Serre}. This is universal for
morphisms from $U$ to semi-abelian varieties. There is an Albanese homomorphism
${\alb}_U \colon H^S_0(U)_0 \to \Alb(U)(k)$, where $H^S_0(U)_0$ is the kernel of
the push-forward map $H^S_0(U)_0 \to H^S_0(\pi_0(U))$.
A famous theorem of Roitman \cite{Roitman}
says that if $U$ is projective (i.e., $D = \emptyset$ in our
set-up), then ${\alb}_U$ induces an isomorphism between the torsion subgroups, away from
$\Char(k)$. The latter condition was subsequently removed by Milne \cite{Milne}. 

Spie{\ss} and Szamuely \cite{SS} showed that, away from $\Char(k)$,
${\alb}_U$ induces an isomorphism between the torsion subgroups even if $D \neq \emptyset$.
Geisser \cite[Thm.~1.1]{Geisser} showed that the condition imposed by Spie{\ss}-Szamuely could be removed if
one assumed resolution of singularities. Recently, Ghosh-Krishna
\cite[Thm.~1.7]{Ghosh-Krishna-CFT}
showed that Geisser's condition could be eliminated. But their proof is long and
intricate. Using Theorem \ref{thm:Main}, we can give a very quick proof (see
\S~\ref{sec:LW-pf}) of (the unconditional version of)
the torsion theorem of Spie{\ss}-Szamuely
in positive characteristic. The result is the following.

\begin{thm}\label{thm:RTT}
  Let $(X,D)$ be a reduced modulus pair over $k$ and $U = X \setminus D$. Then the
  Albanese map for $U$ induces an isomorphism
  \[
    \alb_U \colon   H^S_0(U)_\tor \xrightarrow{\cong} \Alb(U)(k)_\tor.
  \]
\end{thm}

\subsubsection{Motivic cohomology of normal crossing schemes}\label{sec:MCSNC}
Let $k$ be a field and let $X$  be a reduced quasi-projective $k$-scheme. Let $H^m(X, \Z(n))$
denote the Friedlander-Voevodsky motivic cohomology of $X$ (see \S~\ref{sec:MC*}).
This is an abstractly defined cohomology theory for $X$
which is homotopy-invariant. If $X$ is smooth over $k$
of pure dimension $d$, then it is well known that there is a canonical isomorphism
$\CH_0(X) \xrightarrow{\cong} H^{2d}(X, \Z(d))$. This is a special case of a more general
result of Voevodsky,  that identifies  the motivic cohomology groups of smooth schemes over any field  with the higher Chow groups as defined by  Bloch. See
\cite[Corollary 2]{VoevodskyHigherChow}.

As one knows, the cohomological analogue of $\CH_0(X)$ is
the Levine-Weibel Chow group $\CH^{LW}_0(X)$ when $X$ is singular.
We let $\Lambda$ be a commutative ring which is $\Z$ if $k$ admits resolution of
singularities and is any $\Z[\tfrac{1}{p}]$-algebra if $\Char(k) = p > 0$.
The following is an open question in the theory of algebraic cycles.

\begin{ques}\label{ques:Open}
  Let $X$ be a seminormal{\footnote{If
    one wants to replace $\Lambda$ by $\Z$, then one should replace seminormal by
    weakly normal.}} quasi-projective $k$-scheme of pure dimension $d$. Is there a
  canonical isomorphism
  \[
    \CH^{LW}_0(X)_\Lambda \to H^{2d}(X, \Lambda(d))?
  \]
  \end{ques}

We do not know if this question may have a positive answer. We can however prove the
  following result using Theorem \ref{thm:Main}.

Let $\CH^{\lci}_0(X)$ denote the lci version of the Levine-Weibel Chow group of $X$ as
defined in \cite[\S~3]{Binda-Krishna} (see \S~\ref{sec:LWC}).
This is a modified form of $\CH^{LW}_0(X)$ with better functorial properties.
As another application of Theorem \ref{thm:Main}, we can prove the following result with
regard to the above question.

\begin{thm}\label{thm:LW-MC}
  Let $k$ be any field and $X$ a reduced
  quasi-projective scheme of pure dimension $d$ over $k$. Then the following hold.
  \begin{enumerate}
    \item
  There exists a canonical  homomorphism
  \[
    \lambda_X \colon \CH^{\lci}_0(X) \to H^{2d}(X, \Z(d)).
  \]
\item
  $\lambda_X$ is surjective with $\Lambda$-coefficients if $X$ is projective and
  the regular locus of $X$ is smooth over $k$.
\item
  $\lambda_X$ is an isomorphism with $\Lambda$-coefficients if $X$ is a projective normal
  crossing scheme over $k$.
  \end{enumerate}
\end{thm}

The last part of Theorem~\ref{thm:LW-MC} was earlier shown in
\cite[Thm.~1.6]{Binda-Krishna-21} if one assumes
that $\CH^{LW}_0(X)_\Lambda \cong \CH^{\lci}_0(X)_\Lambda$ and one of the following holds.
\begin{enumerate}
\item
  $k$ is infinite and perfect of positive characteristic.
  \item
${\rm char}(k) = 0$ and $\Lambda = {\Z}/m, m \neq 0$.
\end{enumerate}
The assumption $\CH^{LW}_0(X)_\Lambda \cong \CH^{\lci}_0(X)_\Lambda$ is usually very hard to check, even though it is unavoidable in \cite{Binda-Krishna-21}. Note that it is automatically satisfied if e.g. ${\rm char}(k)=0$ and $k$ is algebraically closed. We refer the reader to \cite[Lemma 8.1]{Binda-Krishna-21} and the references in \emph{loc. cit.} for a more detailed comparison.

\subsubsection{A question of Barbieri-Viale and Kahn}
Let $k$ be an algebraically closed field of characteristic zero. As an application of the comparison between the Levine-Weibel Chow group of zero cycles and the $(2d,d)$ motivic cohomology group, we can give a positive answer to a question posed by Barbieri-Viale and Kahn in \cite{BVKahn}. This can be interpreted as a comparison between the  Roitman theorem for the cdh-motivic cohomology, proved in \cite{BVKahn}, and the more classical Roitman theorem for singular projective varieties in characteristic zero, proved in \cite{BiswasSrinivas}. 

\begin{thm}\label{thm:BVKahnintro} Let $k$ be an algebraically closed field  of characteristic zero and let $X$ be a reduced projective $k$-scheme of pure dimension $d$. Then the morphism
    \[\lambda_X\colon \CH_0^{LW}(X) \to H^{2d}(X, \Z(d))\] is surjective with uniquely divisible kernel, and  there is a commutative diagram 
    \[\begin{tikzcd}
  \CH_0^{LW}(X)_{\rm tors} \arrow[r, "\lambda_X"] \arrow[d, "a^+"]& H^{2d}(X, \Z(d))_{\rm tors} \arrow[d, "u"] \\
   \Alb^+(X)(k)_{\rm tors} \arrow[r] & \mathbf{L}_1\Alb^*(X)(k)_{\rm tors},
    \end{tikzcd}
    \]
    where all the arrows are isomorphisms. Here, $\Alb^+(X)(k)$ is the universal semi-abelian regular quotient of $\CH_0^{\lci}(X)$, and $\mathbf{L}_1\Alb^*(X)(k)$ is the semi-abelian part of the $1$-motive ${\rm LAlb}(M(X)^*(d)[2d])$.
\end{thm}

\vskip .3cm

We end the discussion of our main result and its application with the following
larger question. Let $\Lambda$ be as in \S~\ref{sec:MCSNC}.

\begin{ques}\label{ques:Conj}
  Let $k$ be any field. Let $X$ be a smooth projective $k$-scheme of pure dimension $d$
  and let $D \subset X$ be a reduced simple normal crossing divisor. Let
  $U = X \setminus D$. Is there a canonical isomorphism
  \[
    \CH^m(X|D,n)_\Lambda \xrightarrow{\cong} H^{2m-n}_c(U, \Lambda(m))?
  \]
  \end{ques}

\vskip .3cm

 \subsection{Overview of proofs}\label{sec:Outline}
 We prove Theorem \ref{thm:Main} by induction on $\dim(X)$. This reduces the proof to the
 case when $X$ is a surface. The case of surfaces is the most delicate one and the main work
 goes into proving this case. The main steps are as follows.
 
 We use the decomposition theorem of \cite{GKR} as first of the key tools.
  This result
  provides an injective homomorphism $p_* \colon \CH_0(X|D) \to \CH^{\lci}_0(S_X)$,
  where $S_X$ is the double of $X$ along $D$. The proof of Theorem \ref{thm:Main} 
  is then essentially equivalent to showing that $p_*$ factors through the quotient
  $\CH_0(X|D) \surj H^S_0(U)$.
  The second step is to show that if we compose $p_*$ with
  the pull-back $\CH^{\lci}_0(S_X) \to \CH^{\lci}_0(S^{sn}_X)$, then $p_*$ does factor
  through $H^S_0(U)$, where $S^{sn}_X$ is the seminormalization of $S_X$.
  The third step is to show that this pull-back map is an isomorphism (under
  the given assumptions on $k$ and $D$). To show the latter, 
  we prove some results that compare Quillen's algebraic $K$-theory and Weibel's
  homotopy $KH$-theory for certain types of curves and surfaces.

  Most of the applications given above are immediate consequences of
  Theorem~\ref{thm:Main}, with the exception of Theorem~\ref{thm:LW-MC}.
  To prove Theorem~\ref{thm:LW-MC}, we proceed as follows.
We first construct the
map $\lambda_X$ using the Gysin maps for Chow group and motivic cohomology.
This reduces the construction to dimension one case which we
deduce using the slice spectral sequence for singular schemes from \cite{Krishna-Pelaez}.
  The key idea then is to replace $\CH^{\lci}_0(X)$ with a cycle group
  $\CH^{EKW}_0(X)$, introduced by Esnault-Kerz-Wittenberg \cite{EKW}.
This is possible, thanks to Theorem \ref{thm:Main}.
We then use a result of Cisinski-D{\'e}glise \cite{Cisinski-Deglise}
on the perfection properties of various cycle groups
to pass to a perfect base field. Theorem \ref{thm:LW-MC} then follows.

  In \S~\ref{sec:K-thry}, we collect the $K$-theoretic results that we need to prove
  Theorem \ref{thm:Main} for surfaces. In \S~\ref{sec:Pf-main}, we prove the key
  factorization lemma which allows us to conclude the proof.
  We also prove Theorem~\ref{thm:RTT} in this section.
  We prove Theorem \ref{thm:LW-MC} in \S~\ref{sec:NCS}.
  Finally, \S~\ref{sec:BVKahn} is dedicated to the proof of
  Theorem \ref{thm:BVKahnintro}.

\subsection{Notations}\label{sec:Notn}
Throughout this note, we fix a field $k$. A $k$-scheme will
mean a separated and essentially of finite type $k$-scheme.
We shall denote the category of such schemes by $\Sch_k$.
We shall let $\Sm_k$ be the subcategory of $\Sch_k$ consisting of smooth schemes over $k$.
If $X \in \Sch_k$ is reduced, we shall let $X^n$ (resp. $X^{sn}$)
denote the normalization (resp. seminormalization) of $X$.

Recall that for $A$ a reduced commutative Noetherian ring and $B$ a subring of the integral closure of $A$ in its ring of total quotients, which is finite as $A$-module, we say that an ideal $I\subset A$ is a conducting ideal for the inclusion $A\subset B$ if $I=IB$.  More generally, if $f\colon X'\to X$ is a finite birational map, a closed subscheme $Y$ of $X$ is called a conducting subscheme for $f$ if the sheaf of ideals $\mathcal{I}_Y \subset \cO_X$ is a sheaf of conducting ideals for the inclusion of sheaves of rings $\cO_X\to f_* (\cO_{X'})$. 
We shall let $k(X)$ denote the total ring of quotients of $X$.
For a morphism $f \colon X' \to X$ of $k$-schemes and
$D \subset X$ a subscheme, we shall write $D \times_X X'$ as $f^*(D)$.
If $Y, Z \subset X$ are two closed subschemes, then $Y \cap Z$ will mean
the scheme theoretic intersection $Y \times_X Z$ unless we say otherwise.
We shall let $\sZ_0(X)$ denote the free abelian group on the set of closed points on
$X$.

\vskip .4cm

\noindent\emph{Acknowledgements.}
The authors would like to thank Bruno Kahn for providing some comments on a preliminary version of this manuscript. They would also like to thank the referee for
reading the manuscript carefully and providing helpful suggestions.

\section{Algebraic and homotopy $K$-groups of a double}
\label{sec:K-thry}
The goal of this section is to show that if $S_X$ is the double of a regular
surface $X$ over a field along a reduced Cartier divisor, then $SK_0(S_X)$
coincides with an analogous subgroup of $KH_0(S_X)$ under some necessary conditions on
$k$ and $D$. We shall begin by recollecting necessary concepts.
We shall then prove some preliminary $K$-theoretic results before reaching the goal.
We fix a field $k$ throughout this section.

\subsection{Review of double along a divisor}\label{sec:Double}
Let $X \in \Sch_k$ be a regular scheme and let $D \subset X$ be an effective Cartier
divisor. Recall from \cite[\S~2.1]{Binda-Krishna}
  that the double of $X$ along $D$ is the push-out $S_X:= X \amalg_D X$. One knows that
  \begin{equation}\label{eqn:Double-0}
    \xymatrix@C.8pc{
      D \ar[r]^-{\iota} \ar[d]_-{\iota} & X_+ \ar[d]^-{\iota_+} \\
      X_- \ar[r]^-{\iota_-} & S_X}
  \end{equation}
  is a bi-Cartesian square. Moreover, there is a finite and flat morphism
  $\Delta \colon S_X \to X$ whose composition with $\iota_\pm$ are identity.
  $S_X$ is a reduced Cohen-Macaulay scheme with two irreducible components $X_\pm$
  and its normalization $S^n_X$ is canonically isomorphic to $X_+ \amalg X_-$ (see again \cite[Prop. 2.4]{Binda-Krishna}).
  For the normalization morphism $\pi \colon S^n_X \to S_X$, the smallest conductor
  subscheme inside $S_X$ is $D$ whose inverse image in $S^n_X$ is $D \amalg D$.
  It follows that $S_X$ is a seminormal scheme if $D$ is reduced (see
  \cite[Prop.~4.2]{Krishna-torsion}). 

  If $C_\pm$ are two closed subschemes of $X$ not contained in $D$
  such that $C_+ \cap D = C_- \cap D$ as
  closed subschemes, then the join $C_+ \amalg_D C_-$ along $D \cap C_\pm$ is
  canonically a closed subscheme of $S_X$. If $\nu \colon C \inj X$ is a regular closed
  immersion whose image is not contained in $D$, then the double of $C$ along $C \cap D$
  (which we shall also denote by $S_C$) has the property that
  the inclusion $\nu' \colon S_C \inj S_X$ is also a regular closed
  immersion (again by \cite[Prop.~2.4(5)]{Binda-Krishna}). We shall use this fact
  often in this paper.

\subsection{Review of homotopy $K$-theory}\label{sec:KKH}
  Recall (e.g., see \cite[\S~5]{Krishna-Ravi}) that the homotopy $K$-theory
  spectrum (introduced by Weibel \cite{Weibel-KH})
  of a $k$-scheme is defined as the homotopy colimit spectrum
  $KH(X) = \hocolim_n K(X \times \Delta_n)$, where $\Delta_\bullet$ is the standard
  cosimplicial scheme defined by setting $\Delta_n = \Spec({k[t_0, \ldots , t_n]}/{(\sum_i t_i - 1)})$.
  There is a natural transformation between the presheaves of $S^1$-spectra
  $K(X) \to KH(X)$ on $\Sch_k$, which is a weak equivalence if $X$ is regular.
  Furthermore, if $f \colon X' \to X$ is a proper local
  complete intersection morphism (or, more generally,  a morphism of finite Tor-dimension),
  then so is $f \times {\rm id} \colon X' \times \Delta_\bullet
  \to X \times \Delta_\bullet$. % Since the Cartesian square
 % \[
 %   \xymatrix@C.8pc{
 %     X' \times \Delta_{n} \ar[r] \ar[d] & X' \times \Delta_m \ar[d] \\
 %     X \times \Delta_{n} \ar[r]  & X \times \Delta_m}
 % \]
  %is tor-independent for each transition map $[n] \to [m]$ of $\Delta_\bullet$,
  It follows from \cite[Prop.~3.18]{TT} that there is a push-forward map
  between the simplicial spectra $f_* \colon K(X' \times \Delta_\bullet) \to
  K(X \times \Delta_\bullet)$. Taking the homotopy colimits, we see that there
  is a push-forward map $f_* \colon KH(X') \to KH(X)$. This map satisfies usual properties
  such as the composition law and commutativity with pull-back.

Let $\tau \colon (\Sch_k)_\cdh \to (\Sch_k)_\zar$ be the canonical morphism of sites, where
  $(\Sch_k)_\cdh$ denotes the category $\Sch_k$ equipped with the cdh topology
  (e.g., see \cite[Chap.~12]{MVW}). 
Since $KH(X)$ is homotopy equivalent to the cdh-fibrant replacement of the spectrum $K(X)$
(see \cite{Cisinski}, \cite{Haesemeyer}), there is a commutative diagram of
strongly convergent spectral sequences
\begin{equation}\label{eqn:SS}
  \xymatrix@C1pc{
 E^{p,q}_2 = H^p_{\zar}(X, \sK_{q,X}) \ar[d] &  \Rightarrow & K_{q-p}(X) \ar[d] \\
~'E^{p,q}_2 = H^p_{\cdh}(X, \sK_{q,X}) & \Rightarrow & KH_{q-p}(X),}
      \end{equation}
      where the top one is the Zariski descent spectral sequence due to Thomason-Trobaugh
      \cite[Thm.~10.3]{TT}.

      Let us describe the edge homomorphisms of these spectral sequences in low degrees.
      First,  there is a natural map ${\rm rk} \colon KH_0(X) \to H^0(X, \Z)$ whose composition
      with $K_0(X) \to KH_0(X)$ is the (classically defined) rank map. We let $\wt{K}_0(X)$ and $\wt{KH}_0(X)$
      denote the  respective kernels. Using the above spectral sequences again, we get a natural
      map ${\rm det} \colon \wt{KH}_0(X) \to H^1_\cdh(X, \sO^{\times}_X)$, which is
      surjective
      if $\dim(X) \le 2$. We let $SKH_0(X)$ denote its kernel. We let
      $SK_0(X)$ be the kernel of the (surjective) determinant map ${\rm det} \colon \wt{K}_0(X) \surj
H^1_\zar(X, \sO^{\times}_X) = \Pic(X)$ (\cite[Thm. II.8.1]{Weibel}).

Applying the above spectral sequences to $KH_1(X)$, we get an edge map
$KH_1(X) \to  H^0_\cdh(X, \sO^{\times}_X)$. We let $SKH_1(X)$ denote its kernel.
Similarly, we let $SK_1(X)$ denote the kernel of the edge map
$K_1(X) \to H^0_\zar(X, \sO^{\times}_X)$. 
Let $X^{sn} \to X$ denote the seminormalization morphism when $X$ is reduced
(see \cite[\S~4.1]{Krishna-torsion}). 

\begin{lem}\label{lem:K1-split}
  Let $X \in \Sch_k$ be a reduced scheme. Then we have the following.
\begin{enumerate}
  \item
    The canonical map
  $H^0_\zar(X, \sO^{\times}_X) \to H^0_\cdh(X, \sO^{\times}_X)$ has a factorization
  \[
    \xymatrix@C1pc{
      H^0_\zar(X, \sO^{\times}_X) \ar[r] \ar[d] &
      H^0_\zar(X^{sn}, \sO^{\times}_{X^{sn}}) \ar[d]^-{\cong} \ar@{.>}[dl] \\
      H^0_\cdh(X, \sO^{\times}_X) \ar[r]^-{\cong} &  H^0_\cdh(X^{sn}, \sO^{\times}_{X^{sn}}),}
  \]
  where the horizontal arrows are induced by the projection $X^{sn} \to X$ and
  the vertical arrows are induced by the change of topology.
  Moreover, the bottom horizontal and the right vertical arrows are isomorphisms.

\item
  There is a commutative diagram of short exact sequences
  \begin{equation}\label{eqn:K1-split-0}
    \xymatrix@C1pc{
      0 \ar[r] & SK_1(X) \ar[d] \ar[r] & K_1(X) \ar[r] \ar[d] & H^0_\zar(X, \sO^{\times}_X)
      \ar[r] \ar[d] & 0 \\
   0 \ar[r] & SKH_1(X) \ar[r] & KH_1(X) \ar[r] & H^0_\cdh(X, \sO^{\times}_X)
      \ar[r] & 0,}
  \end{equation}
  which split functorially in $X$.
  \end{enumerate}
\end{lem}
\begin{proof}
The first part is well known (e.g., apply
  \cite[Prop.~6.14]{Huber-Kelly} with $Y = \G_m$).
  It is shown in \cite[Lem.~2.1]{Krishna-ant} that the top sequence
  is split exact such that the splitting is functorial in $X$.
  The bottom sequence is left exact by definition. We now show that it is actually
  split exact.
  
  Using \cite[Prop.~3.2]{Weibel-KH} and
  Zariski descent for $KH$-theory, it follows that the canonical map
  $KH(X) \to KH(X^{sn})$ is a weak equivalence. Hence, we can assume $X$ to be
  seminormal to prove the split exactness of the bottom sequence in
  ~\eqref{eqn:K1-split-0}. Using the split exact property of the top sequence, it suffices
  to show that the change of topology map $H^0_\zar(X, \sO^{\times}_X) \to
  H^0_\cdh(X, \sO^{\times}_X)$ is an isomorphism. But this is the first part of the lemma.
\end{proof}

The spectral sequences of ~\eqref{eqn:SS} imply that there is a commutative
  diagram of short exact sequences
  \begin{equation}\label{eqn:SS-curve}
    \xymatrix@C1pc{
      0 \ar[r] & H^1_\zar(X, \sK_{i+1,X}) \ar[r] \ar[d] &  K_i(X) \ar[r] \ar[d] &
      H^0_\zar(X, \sK_{i,X}) \ar[r] \ar[d] &  0 \\
      0 \ar[r] & H^1_\cdh(X, \sK_{i+1,X}) \ar[r] &  KH_i(X) \ar[r] &
      H^0_\cdh(X, \sK_{i,X}) \ar[r] & 0}
  \end{equation}
 for every $i \ge 0$ if $\dim(X) \le 1$.
  Combining Lemma~\ref{lem:K1-split} and ~\eqref{eqn:SS-curve}, we get
  \begin{equation}\label{eqn:SS-curve-*}
  SK_1(X) \cong  H^1_\zar(X, \sK_{2,X}) \ \ {\rm and} \ \ 
  SKH_1(X) \cong  H^1_\cdh(X, \sK_{2,X})
\end{equation}
if $\dim(X) \le 1$.

For a closed immersion $W \subset Z$ in $\Sch_k$, we let $K(Z,W)$ be the relative homotopy
$K$-theory spectrum of the pair $(Z,W$).  It is defined as the  homotopy fiber
of the restriction map of spectra $K(Z) \to K(W)$. If $f \colon Z' \to Z$ is a morphism
of $k$-schemes such that $W' = W \times_Z Z'$, we let $K(Z, Z', W)$ (the double relative
$K$-theory spectrum) denote the homotopy fiber
of the canonical pull-back map $f^* \colon K(Z,W) \to K(Z', W')$. We define
$KH(Z,W)$ and $KH(Z,Z',W)$ in analogous fashion.

\subsection{Algebraic and homotopy $K_2$-groups of normal crossing
  curves}\label{sec:curves*}
We shall now compare $K_2(X)$ and $KH_2(X)$ when $X$ is a normal crossing curve.
We first recall the definition of normal crossing schemes that we shall use in this
paper.

% \subsection{Normal crossing schemes}\label{sec:SNC}
Let $X \in \Sch_k$ be a reduced scheme of pure dimension $d \ge 0$.
Let $\{X_1, \ldots , X_n\}$ be the set of irreducible components of $X$. We 
shall say that $X$ is a normal crossing $k$-scheme 
if for every nonempty subset $J \subset [1, n]$, the scheme theoretic 
intersection $X_J := {\underset{i \in J}\bigcap} X_i$ is either empty or
a smooth $k$-scheme of pure dimension $d + 1 - |J|$.
Recall that $X \in \Sch_k$ is called $K_i$-regular if the map
  $K_i(X) \to K_i(X \times \Delta_n)$, induced by the projection, is an isomorphism for all
  $n \ge 0$.

\begin{lem}\label{lem:K1-reg}
  Let $X \in \Sch_k$ be a normal crossing curve. Then $X$ is $K_i$-regular for
  $i \le 1$.
\end{lem}
\begin{proof}
  It is well known (e.g., use the Bass fundamental exact sequence)
  that the lemma is equivalent to the assertion that $X$ is $K_1$-regular.
  We first assume that $X$ is affine.
  We let $\mu(X)$ denote the number of irreducible components of $X$ and
  write $X(n) = X \times \Delta_n$.
 We shall prove $K_1$-regularity of $X$ by induction on $\mu(X)$. The case $\mu(X) = 1$
  is trivial because $X$ is then smooth, and one knows that
  smooth (more generally, regular) schemes are $K_i$-regular for all $i$.
  Let us now assume that $\mu(X) > 1$.
  We let $X_1$ be an irreducible component of $X$ and let $X_2$ be the scheme theoretic
  closure of $X \setminus X_1$ in $X$. We let $Y = X_1 \cap X_2$. Then
  $X_2$ is a normal crossing curve such that $\mu(X_2) = \mu(X) -1$ and
  $Y$ is a 0-dimensional smooth $k$-scheme.
  We have a commutative square
  \begin{equation}\label{eqn:Milnor-square}
    \xymatrix@C1pc{
      Y(n) \ar[r] \ar[d] & X_1(n) \ar[d] \\
    X_2(n) \ar[r] & X(n)}
\end{equation}
of affine schemes
for every $n \ge 0$, which is Cartesian as well as co-Cartesian and in which all arrows
are closed immersions.

By \cite[Thm.~6.4]{Milnor}, there exists a commutative diagram of
exact sequences
\begin{equation}\label{eqn:Milnor-square-0}
  \xymatrix@C.8pc{
 K_2(Y) \ar[r] \ar[d] & K_1(X)
 \ar[r] \ar[d]^-{p^*_X} &  K_1(X_1) \oplus K_1(X_2) \ar[r] \ar[d] & K_1(Y) \ar[d] \\
 K_2(Y(n)) \ar[r] & K_1(X(n))
  \ar[r] &  K_1(X_1(n)) \oplus K_1(X_2(n)) \ar[r]  & K_1(Y(n)),}
\end{equation}
in which the vertical arrows are induced by the projection maps.
The left-most and the right-most vertical arrows are isomorphisms because $Y$ is smooth.
The vertical arrow involving $X_1$ and $X_2$ is an isomorphism by induction on $\mu(X)$.
It follows via a diagram chase that $p^*_X$ is surjective. As this map is always
(split) injective,  the affine case of the lemma follows.

If $X$ is not necessarily affine, we choose a dense open affine subscheme $U \subset X$
such that $X_\sing \subset U$ and let $Y = X \setminus U$ with the reduced induced
closed subscheme structure. Then $Y$ is a regular $k$-scheme.
Using the Thomason-Trobaugh localization sequence (\cite[Thm.~7.4]{TT})
\[
  K^{Y(n)}(X(n)) \to K(X(n)) \to K(U(n))
\]
and the weak equivalence $K(Y(n)) \xrightarrow{\sim} K^{Y(n)}(X(n))$
(this uses excision and the fact that $Y = Y_\reg \subset X_\reg$),
we get a commutative diagram of exact sequence of homotopy groups
\begin{equation}\label{eqn:Milnor-square-1}
  \xymatrix@C.8pc{
K_1(Y) \ar[r] \ar[d] & K_1(X) \ar[r] \ar[d]^-{p^*_X} & K_1(U) \ar[r] \ar[d]^-{p^*_U} & 
K_0(Y) \ar[d] \\
K_1(Y(n)) \ar[r] & K_1(X(n)) \ar[r] & K_1(U(n)) \ar[r] & 
K_0(Y(n)).}
\end{equation}

The left-most and the right-most vertical arrows are isomorphisms because $Y$ is regular.
The arrow $p^*_U$ is an isomorphism because $U$ is affine.
It follows via a diagram chase that $p^*_X$ is surjective.  
As this map is always (split) injective,  the lemma follows.
\end{proof}

\begin{lem}\label{lem:NC-K2}
  Let $X \in \Sch_k$ be a normal crossing curve. Then the canonical map
  $K_2(X) \to KH_2(X)$ is surjective. In particular, the map
  $H^0_\zar(X, \sK_{2,X}) \to H^0_\cdh(X, \sK_{2,X})$ is surjective.
\end{lem}
\begin{proof}
  It follows from the spectral sequences \eqref{eqn:SS} that the
  edge maps $K_2(X) \to H^0_\zar(X, \sK_{2,X})$ and $KH_2(X) \to H^0_\cdh(X, \sK_{2,X})$
  are surjective. Hence, we only need to prove the first assertion of the lemma.
  In view of Lemma~\ref{lem:K1-reg}, the spectral sequence
  \[
    E^{p,q}_1 = K_q(X \times \Delta_p) \Rightarrow KH_{p+q}(X),
  \]
  degenerates to an exact sequence
  \begin{equation}\label{eqn:NC-K2-0}
    K_2(X \times \Delta_1) \xrightarrow{\partial^*_1 - \partial^*_0}
  K_2(X) \to KH_2(X) \to 0.
\end{equation}
This finishes the proof.
\end{proof}

\vskip .3cm

\subsection{Algebraic and homotopy $K_1$-groups of curves}
\label{sec:K1-curve}
For the rest of \S~\ref{sec:K-thry}, we shall work with the following set-up.
We let $X$ be a regular integral quasi-projective surface over $k$ and $D \subset X$
  a reduced effective Cartier divisor.
 Recall that $S_X$ denotes the double of $X$ along $D$.
The goal of this subsection is to prove the following two lemmas.

\begin{lem}\label{lem:D-SK-1}
    The map $SK_1(D) \to SKH_1(D)$ is an isomorphism under any of the following
    conditions.
    \begin{enumerate}
    \item
      $k$ is perfect and $D$ is seminormal.
    \item
      $k \subseteq \ov{\Q}$.
    \end{enumerate}
  \end{lem}
  \begin{proof}
Let $\pi \colon D^n \to D$ be the normalization morphism. Let $E \subset D$ be
    a conducting closed subscheme for $\pi$ and let $E' = \pi^*(E)$.
    We then have an exact sequence of relative and double relative $K$-groups:
    \begin{equation}\label{eqn:D-SK-ex-0}
      \cdots \to  K_i(D,D^n, E) \to K_i(D,E) \to K_i(D^n, E') \to K_{i-1}(D,D^n,E) \to
      \cdots .
    \end{equation}
Since $\dim(D) = 1$ and $D$ is reduced, the conducting subscheme $E$ is supported on a finite set of closed points (hence it is affine). It follows from \cite[Thm.~0.2]{Geller-Weibel} and
    the Thomason-Trobaugh descent spectral sequence \cite{TT} that $K_0(D,D^n,E) = 0$ and
    $K_1(D,D^n,E) \cong {\sI_E}/{\sI^2_E} \otimes_{E'} \Omega^1_{E'/E}$, where
    $\sI_E$ is the Zariski sheaf of ideals on $D$ defining $E$.

 Assume that $(1)$ holds. Since  $D$ is seminormal,  we can choose our conducting subschemes $E$ and $E'$
 to be reduced (in fact, $E$ can be chosen to be $V(\mathcal{I})$ where $\mathcal{I}$ is the largest conducting ideal for the map $\cO_{D} \subset \pi_* \cO_{D^n}$, see \cite[Prop.~4.2(1)]{Krishna-torsion}). It follows that the coordinate rings of $E$ and $E'$ are finite products of finite and separable extensions of $k$,
 and that the extension $E'/E$ is separable (this uses the perfectness hypothesis).
 It follows that $\Omega^1_{E'/E}=0$, hence
    $K_1(D,D^n,E) = 0$. This gives  a  Mayer-Vietoris exact sequence
    \begin{equation}\label{eqn:D-SK-ex-1}
      K_2(D^n) \oplus K_2(E) \to {K_2(E')} \to SK_1(D) \to SK_1(D^n) \to 0,
      \end{equation}
      where $E'$ is reduced (note that $SK_1(E)= SK_1(E')=0$ since $E$ and $E'$ are
      semilocal).

 If $D$ is not seminormal, the conducting subscheme $E$ cannot be chosen to be reduced.  However, for $m$ sufficiently large, we have that $E \subset m E_{\rm red}$ and $E'\subset mE'_{\rm red}$.  It follows from the above expression of $K_1(D,D^n,E)$ that there
 is some conducting closed subscheme $E \subset D$, having the same support as that
 of the maximal conducting subscheme 
      such that one has an exact Mayer-Vietoris sequence
    \begin{equation}\label{eqn:D-SK-ex-2}
      K_2(D^n) \oplus K_2(E) \to {K_2(E')} \to SK_1(D) \to SK_1(D^n) \to 0.
      \end{equation}  

We are interested in estimating the term $K_2(E')$. We claim  that if (2) holds then $K_2(E') = K_2(E'_{\rm red})$.  Since $D^n$ is a normal curve, the coordinate ring $A$ of $E'$
      is a finite product $A= \prod_{i=1}^r A_i$ of Artinian $k$-algebras $A_i$, 
      each of which is isomorphic to a truncated polynomial ring
      of the form $A_i = {k_i[t]}/{(t^{n_i})}$, where $k_i/k$ is a finite field extension
      and $n_i \ge 1$ is an integer (this follows, for example, from Cohen's structure theorem).  In order to prove the claim, we can clearly assume that $r=1$.  Let $k'$ be the residue field of $A$.
     
Let $J$ be the kernel of the augmentation ideal $A\to k'$ and write $\Omega^1_{(A, J)}$ for the kernel of the map $\Omega^1_{A/\mathbb{Z}} \to \Omega^1_{k'/\Z}$.  If $k$ is a $\Q$-algebra, a result of Bloch (see, e.g.,  \cite[Thm.  4.1]{KSri}) gives an isomorphism
      \[ K_2(A, J) \cong HC_1(A ,J) \cong {\Omega^1_{(A, J)}}/{d(J)}.\]
  By computing these groups for
      truncated polynomial rings, we conclude that if  $k \subset \ov{\Q}$,  then $K_2(A, J)=0$. Since the map $K_2(A)\to K_2(A/J)=K_2(k')$ is anyway surjective, we conclude that $K_2(A)=K_2(A/J)=K_2(k')$ in this case. This proves the claim.
   %   $K_2(E', E'_\red) \cong \Omega^1_{k/{\Z}} \otimes_k J$. Since $k \subset \ov{\Q}$,
     % it follows that $K_2(E', E'_\red) = 0$. 
      It follows from the claim that ~\eqref{eqn:D-SK-ex-2} is of the
      form
\begin{equation}\label{eqn:D-SK-ex-3}
      K_2(D^n) \oplus K_2(E) \to {K_2(E'_\red)} \to SK_1(D) \to SK_1(D^n) \to 0.
    \end{equation}

We now compare  the sequences \eqref{eqn:D-SK-ex-1} and \eqref{eqn:D-SK-ex-3} with the corresponding ones for $KH$. Since $KH$-theory   satisfies $cdh$-descent (see \cite{Cisinski}, \cite{Haesemeyer}),
      we always have an exact sequence (see \cite[Cor. IV. 12.6]{Weibel})
\begin{equation}\label{eqn:D-SK-ex-4}
      KH_2(D^n) \oplus KH_2(E_{\red}) \to {KH_2(E'_\red)} \to SKH_1(D) \to SKH_1(D^n) \to 0.
      \end{equation}  
      Putting things together, we get a commutative diagram (in both cases (1) and (2))
      with exact rows
 \begin{equation}\label{eqn:D-SK-ex-5}     
   \xymatrix@C.8pc{
     K_2(D^n) \oplus K_2(E) \ar[r] \ar[d] & {K_2(E'_\red)} \ar[r] \ar[d] &
     SK_1(D) \ar[r] \ar[d] & SK_1(D^n) \ar[r] \ar[d] &  0 \\
     KH_2(D^n) \oplus KH_2(E_\red) \ar[r] & {KH_2(E'_\red)} \ar[r] &  SKH_1(D) \ar[r] &
     SKH_1(D^n) \ar[r] & 0.}
   \end{equation}

The left vertical arrow is surjective, since $K_2(D^n)=KH_2(D^n)$ (because $D^n$ is regular) and $K_2(E)\to K_2(E_{\rm red}) = KH_2(E_{\rm red})$ is surjective because $E$ is semi-local.  The  second vertical arrow (from left) 
is an isomorphism, since $E'_{\rm red}$ is regular. Similarly, the right vertical arrow is an isomorphism. A diagram chase now finishes the proof.
\end{proof}

\begin{lem}\label{lem:D-SK-1-ex}
  The map $SK_1(D) \to SKH_1(D)$ is an isomorphism if $k$ is infinite and $D$ is a normal
  crossing curve.
\end{lem}
\begin{proof}
By \eqref{eqn:SS-curve-*},  the lemma is equivalent to showing that the map   
\begin{equation}\label{eqn:D-SK-1-0}
H^1_\zar(D, \sK_{2,D}) \to H^1_\cdh(D, \sK_{2,D})
  \end{equation}
  is an isomorphism.
Let $\mu(D)$ denote the number of irreducible components of $D$.
  We shall prove the above isomorphism of cohomology groups by induction on $\mu(D)$.
  If $\mu(D) = 1$, then $D$ is regular in which case ~\eqref{eqn:D-SK-1-0} is clear,
  since $K_1(D)=KH_1(D)$.
  Otherwise, we let $D_1$ be an irreducible component of $D$ and let $D_2$ be the
  scheme theoretic closure of $D \setminus D_1$ in $D$. Then $D_2$ is a normal
  crossing curve such that $\mu(D_2) < \mu(D)$ and $D_3 := D_1 \cap D_2$ is a
  0-dimensional smooth $k$-scheme.

We now consider the commutative diagram  
  of Zariski sheaves
  \begin{equation}\label{eqn:D-SK-1-1}
    \xymatrix@C.8pc{
      0 \ar[r] & {\sK_{2, (D,D_2)}} \ar[r] \ar[d] & \sK_{2, D} \ar[r] \ar[d] &
      \sK_{2, D_2} \ar[r] \ar[d] & 0 \\
      0 \ar[r] & {\sK_{2, (D_1,D_3)}} \ar[r] & \sK_{2, D_1} \ar[r] &
      \sK_{2, D_3} \ar[r] & 0,}
    \end{equation}
    where the  terms ${\sK_{2, (D,D_2)}}$ and ${\sK_{2, (D_1,D_3)}}$  in the first and the second row are defined to be the kernels of the  right horizontal arrows. 
    Since $k$ is infinite, the Quillen $\sK_2$-sheaf coincides with the Milnor
    $\sK^M_2$-sheaf on the big Zariski site of $\Sch_k$ (see \cite[Prop.~2]{Kerz-JAG}).      In particular, sections of $\mathcal{K}_2$ are given by symbols, thus the two right horizontal arrows of \eqref{eqn:D-SK-1-1} are indeed surjective and thus both rows in \eqref{eqn:D-SK-1-1} are exact.
    
By \cite[Prop.~2]{Kerz-JAG} again, the terms ${\sK_{2, (D,D_2)}}$ and ${\sK_{2, (D_1,D_3)}}$  coincide with the relative Milnor $K$-sheaves of Kato-Saito \cite[1.3]{Kato-Saito-86}. In particular,   the map $\sK^M_{2, (D,D_2)} \to \sK^M_{2, (D_1,D_3)}$ is surjective by
\cite[Lem.~1.3.1]{Kato-Saito-86}. Furthermore, the kernel of this map is supported on
$D_3$. It follows that the left vertical arrow in the above diagram is surjective whose
kernel is supported on $D_3$. In particular, the induced map between the first Zariski
cohomology groups is an isomorphism.

Using a diagram chase of the cohomology groups induced by ~\eqref{eqn:D-SK-1-1}, we
    therefore get a Mayer-Vietoris exact sequence for the Zariski cohomology in low degrees:
   \begin{equation}\label{eqn:D-SK-1-2} 
     {\begin{array}{l}
       H^0_\zar(D_1, \sK_{2,D_1}) \\
       \hspace*{1.2cm} \oplus \\
       H^0_\zar(D_2, \sK_{2,D_2})
     \end{array}}
   \to
     H^0_\zar(D_3, \sK_{2,D_3}) \to H^1_\zar(D, \sK_{2,D}) \to 
     {\begin{array}{l}
       H^1_\zar(D_1, \sK_{2,D_1}) \\
       \hspace*{1.2cm} \oplus \\
        H^1_\zar(D_2, \sK_{2,D_2})
        \end{array}}
       \to 0.
     \end{equation}

We next observe that the square
    \begin{equation}\label{eqn:D-SK-1-3}
         \xymatrix@C.8pc{
           D_3 \ar[r] \ar[d] & D_1 \ar[d] \\
           D_2 \ar[r] & D}
       \end{equation}
       defines a cdh cover $\{D_1 \amalg D_2 \to D\}$ of $D$,
       so that we have an  associated exact sequence  of
       cdh cohomology groups similar to ~\eqref{eqn:D-SK-1-2}.
       Comparing these two sequences, we get a commutative diagram of exact sequences

\begin{equation}\label{eqn:D-SK-1-4}
         \xymatrix@C.8pc{
{\begin{array}{l}
       H^0_\zar(D_1, \sK_{2,D_1}) \\
       \hspace*{1.2cm} \oplus \\
       H^0_\zar(D_2, \sK_{2,D_2})
     \end{array}}
   \ar[r] \ar[d] & 
     H^0_\zar(D_3, \sK_{2,D_3}) \ar[r] \ar[d] & H^1_\zar(D, \sK_{2,D}) \ar[r] \ar[d] & 
     {\begin{array}{l}
       H^1_\zar(D_1, \sK_{2,D_1}) \\
       \hspace*{1.2cm} \oplus \\
        H^1_\zar(D_2, \sK_{2,D_2})
      \end{array}} \ar[r] \ar[d] & 0 \\
 {\begin{array}{l}
       H^0_\cdh(D_1, \sK_{2,D_1}) \\
       \hspace*{1.2cm} \oplus \\
       H^0_\cdh(D_2, \sK_{2,D_2})
     \end{array}}
   \ar[r] & 
     H^0_\cdh(D_3, \sK_{2,D_3}) \ar[r]  & H^1_\cdh(D, \sK_{2,D}) \ar[r] & 
     {\begin{array}{l}
       H^1_\cdh(D_1, \sK_{2,D_1}) \\
       \hspace*{1.2cm} \oplus \\
        H^1_\cdh(D_2, \sK_{2,D_2})
      \end{array}} \ar[r]  & 0.}
\end{equation}

The left vertical arrow is surjective by Lemma~\ref{lem:NC-K2}. The second vertical
arrow (from left) is an isomorphism because $D_3$ is smooth. The right vertical arrow
is an isomorphism by induction on $\mu(D)$. By a diagram chase, it follows that
~\eqref{eqn:D-SK-1-0} is an isomorphism. This concludes the proof of the
lemma.
\end{proof}

\subsection{Algebraic and homotopy $K_0$-groups of the double}
\label{sec:curves}
We continue with the set-up described in \S~\ref{sec:K1-curve}.
In this subsection, we shall compare $SK_0(S_X)$ with the analogous
subgroup of $KH_0(S_X)$. Let $\pi \colon S^n_X \to S_X$ denote the normalization map.

\begin{lem}\label{lem:Excision-K}
      There exists an exact sequence
      \[
        0 \to \frac{SK_1(S^n_X)}{SK_1(S_X)} \to SK_1(D) \to SK_0(S_X) \xrightarrow{\pi^*}
      SK_0(S^n_X) \to 0.
    \]
  \end{lem}
\begin{proof} This is a consequence of (the proof of)
    \cite[Thm. 3.3]{BPW} or  \cite[Prop.~2.7]{Gupta-Krishna-BF}, noting that we can choose $Y$ (in the notation of op. cit.) to be $D$.    The claimed exact sequence exists if excision holds for the $K_0$. The obstruction for this excision
    is controlled by ${\sI_D}/{\sI^2_D} \otimes_{D'} \Omega^1_{D'/D}$.  As $D' = D \amalg D$, this term vanishes.
\end{proof}

\begin{remk}\label{remk:Non-reduced}
  It is worth noting that the proof of Lemma \ref{lem:Excision-K} did not use our
  assumption that $D$ is reduced. Hence, the lemma remains valid for any
  effective Cartier divisor $D$.
  \end{remk}

  The analogue of Lemma~\ref{lem:Excision-K} also holds for the $KH$-groups
  by the cdh-descent, as we show now. We consider the abstract blow-up square
\begin{equation}\label{eqn:K-KH-0}
        \xymatrix@C.8pc{
          D \amalg D \ar[r] \ar[d] & S^n_X \ar[d]^-{\pi} \ar[r]^-{\cong} & X_+ \amalg X_- 
          \ar[dl] \\
          D \ar[r] & S_X}
      \end{equation}
Applying the spectral sequence ~\eqref{eqn:SS} and the cdh-excision \cite[IV.12.6]{Weibel} to this abstract
      blow-up square, we get a commutative diagram of exact sequences
      \begin{equation}\label{eqn:Ex-KH-0}
        \xymatrix@C.8pc{
          KH_1(S_X) \ar[r] \ar[d] &  KH_1(S^n_X) \ar[r] \ar[d] & KH_1(D) \ar[r] \ar[d] &
          \wt{KH}_0(S_X) \ar[r] \ar[d] & \wt{KH}_0(S^n_X) \ar[d] \\
          H^0_\cdh(S_X, \sO^{\times}_{S_X}) \ar[r] & H^0_\cdh(S^n_X, \sO^{\times}_{S^n_X}) \ar[r] &
          H^0_\cdh(D, \sO^{\times}_{D}) \ar[r] &
          H^1_\cdh(S_X, \sO^{\times}_{S_X}) \ar[r] &  H^1_\cdh(S^n_X, \sO^{\times}_{S^n_X}).}
        \end{equation}

        It follows from Lemma~\ref{lem:K1-split} that the first three vertical arrows from
        left in ~\eqref{eqn:Ex-KH-0} are surjective. 
The map $H^0_\cdh(S_X, \sO^{\times}_{S_X}) \to H^0_\cdh(S^n_X, \sO^{\times}_{S^n_X})$ is
 clearly injective because $S_X$ is reduced and $S^n_X \to S_X$ is a cdh cover.
Using a diagram chase and taking the kernels of the vertical arrows,
       we get an exact sequence
      \begin{equation}\label{eqn:Ex-KH-1} 
        0 \to \frac{SKH_1(S^n_X)}{SKH_1(S_X)} \to SKH_1(D) \to SKH_0(S_X) \to SKH_0(S^n_X).
      \end{equation}

      \vskip .3cm
      
The main result of \S~\ref{sec:K-thry} is the following.

\begin{prop}\label{prop:K-KH}
The map $SK_0(S_X) \to SKH_0(S_X)$  is an isomorphism under any of the following
    conditions.
    \begin{enumerate}
    \item
      $k$ is perfect and $D$ is seminormal.
    \item
      $k \subseteq \ov{\Q}$.
    \item
     $k$ is infinite and $D$ is a normal crossing curve.
    \end{enumerate}
  \end{prop}
  \begin{proof}
    A comparison of ~\eqref{eqn:Ex-KH-1}  with the exact sequence of Lemma
    \ref{lem:Excision-K} gives rise to a commutative diagram of exact sequences
\begin{equation}\label{eqn:Excision-K-KH}
  \xymatrix@C.8pc{
    0 \ar[r] & \frac{SK_1(S^n_X)}{SK_1(S_X)} \ar[r] \ar@{->>}[d] & SK_1(D) \ar[r]
    \ar[d]^-{\cong} &
  SK_0(S_X) \ar[r]^-{\pi^*} \ar[d] & SK_0(S^n_X) \ar[r] \ar[d]^-{\cong} &  0 \\
  0 \ar[r] & \frac{SKH_1(S^n_X)}{SKH_1(S_X)} \ar[r] & SKH_1(D) \ar[r] & SKH_0(S_X)
  \ar[r]^-{\pi^*} & SKH_0(S^n_X) \ar[r] & 0.}
\end{equation}

Note that $\pi^*$ on the bottom is surjective because the same holds for the
corresponding arrow on the top and the right vertical arrow is an isomorphism by the
regularity of the scheme $S^n_X$. The latter also implies that the left vertical arrow is
surjective. The vertical arrow involving $D$ is an isomorphism by Lemmas
\ref{lem:D-SK-1} and ~\ref{lem:D-SK-1-ex}.
The desired assertion now follows by a diagram chase.
\end{proof}

\section{Proof of the main result}\label{sec:Pf-main}
In this section, we shall prove our main result Theorem \ref{thm:Main}. We shall
also give proofs of some of its applications.
We begin by recalling the definitions of the Chow group of 0-cycles with modulus and
Suslin homology. To prove Theorem \ref{thm:Main}, we shall use two other 0-cycle groups,
namely, the Levine-Weibel Chow group and its modified version called the lci Chow group of
the double. We shall recall these too. We fix a field $k$.

\subsection{Review of Chow group with modulus and Suslin homology}
\label{sec:CHMS}
Let $X$ be an integral quasi-projective $k$-scheme of dimension $d \ge 1$ and let
$D \subset X$ be an effective Cartier divisor. Let $j \colon U \inj X$ be the inclusion of
the complement of $D$ in $X$. Assume that $U$ is regular.
Recall from \cite[\S~1]{Kerz-Saito-Duke} that the Chow group of
0-cycles on $X$ with modulus $D$ is the quotient of $\sZ_0(U)$ by the subgroup $\sR_0(X|D)$
generated
by $\nu_*(\divf(f))$, where $\nu \colon C \to X$ is a finite (and birational to its image)
morphism from an integral normal curve $C$ whose image is not contained in $D$ and $f \in
\Ker(\sO^{\times}_{C, \nu^{-1}(D)} \surj \sO^{\times}_{\nu^*(D)})$. This group is denoted by
$\CH_0(X|D)$.

Recall that the Suslin-Voevodsky singular homology $H^{S}_n(U)$ of $U$
(also called Suslin homology in the literature)  is defined as the $n$-th homology of a
certain explicit complex of algebraic cycles, introduced by Suslin and
Voevodsky \cite{SuslinVoevSingular}.
We do not need to recall this complex. Instead, we shall use the following
equivalent definition of $H^S_0(U)$ in this paper. This equivalence was shown by
Schmidt \cite[Thm.~5.1]{Schmidt-ant}.

\begin{lem}\label{lem:Schmidt}
Assume that $X$ is projective.  Then $H^S_0(U)$ is canonically isomorphic to
the quotient of $\sZ_0(U)$ by the subgroup $\sR^S_0(U)$ generated
by $\nu_*(\divf(f))$, where $\nu \colon C \to X$ is a finite (and birational to its image)
morphism from an integral
normal curve $C$ whose image is not contained in $D$ and $f \in
\Ker(\sO^{\times}_{C, \nu^{-1}(D)} \surj \sO^{\times}_{\nu^*(D)_\red})$. 
\end{lem}

It is clear that there is a canonical surjection
\begin{equation}\label{eqn:Surj-0}
  \phi_{X|D} \colon \CH_0(X|D) \surj H^S_0(U).
\end{equation}
From the definition one gets immediately that $\phi_{X|D}$ is an isomorphism if $X$ is of
dimension $1$ and $D$ is reduced. 
It was shown in \cite[Thm.~4.4]{Binda-Krishna-21} that
$\phi_{X|D}$ may have a non-trivial kernel, even if $D$ is reduced and $k$ is algebraically
closed as soon as $\dim(X)\geq 2$, so that the relationship between the two objects is
quite subtle.  This relationship is the main object of study in this paper.

\subsection{Review of Levine-Weibel and lci Chow groups}\label{sec:LWC}
Let $X$ be an equidimensional reduced quasi-projective $k$-scheme of dimension $d \ge 1$.
Let $X_\sing$ denote the singular locus of $X$ with reduced closed subscheme structure and
let $X_\reg$ denote the complement of $X_\sing$ in $X$.
Let $C \subset X$ be a curve (i.e., an equidimensional one-dimensional $k$-scheme).
Recall (see \cite[\S~1]{Levine-Weibel}) that $C$ is called a Cartier curve on $X$ 
if no component of $C$ lies in $X_\sing$, no embedded point of $C$ lies away from
$X_\sing$,  $\sO_{C, \eta}$ is a field if $\ov{\{\eta\}}$ is a component of $C$
disjoint from $X_\sing$ and, $C$ is defined by a regular sequence at every point of
$C \cap X_\sing$. 
We let $k(C, C\cap X_\sing)^{\times}$ be the image of the natural map
$\sO^{\times}_{C,S} \to \stackrel{s}{\underset{i =1}\oplus} \sO^{\times}_{C, \eta_i}$,
where $\{\eta_1, \ldots , \eta_s\}$ is the set of generic points of $C$ and
$S$ is the union of the closed subset 
$C \cap X_\sing$ and the set of generic points $\eta_i$ of $C$ 
such that $\ov{\{\eta_i\}}$ is disjoint from $X_\sing$.

For $f \in k(C, C\cap X_\sing)^{\times}$, we let $\divf(f) =
\stackrel{s}{\underset{i =1}\sum} \divf(f_i)$, where
$f_i$ is the projection of $f$ onto $\sO^{\times}_{C, \eta_i}$, and
$\divf(f_i)$ is the divisor of the restriction of $f_i$ to the maximal Cohen-Macaulay 
subscheme $C_i$ of $C$ supporting $\eta_i$.
If $C$ is reduced, then $k(C, C\cap X_\sing)^{\times} = \sO^{\times}_{C,S}$ and
for $f \in \sO^{\times}_{C,S}$, $\divf(f)$  is the sum of
$\divf(f_i)$, where the sum runs through
the divisors (in the classical sense, see \cite[Chap.~1]{Fulton})
of the restrictions of $f$ to the components of $C$.
The Levine-Weibel Chow group $\CH^{LW}_0(X)$ is
the quotient of $\sZ_0(X_\reg)$ by the subgroup
$\sR^{LW}_0(X)$ generated by $\divf(f)$, where $f \in k(C, C\cap X)_\sing)^{\times}$ for
a Cartier curve $C$ on $X$.

One says that $C$ a good curve (relative to $X_\sing$)
if it is reduced and there is a finite local complete
intersection (lci) morphism $\nu \colon C \to X$ such that $\nu^{-1}(X_\sing)$ is nowhere
dense in $C$. The lci Chow group of 0-cycles $\CH^{\lci}_0(X)$ is the quotient of
$\sZ_0(X_\reg)$ by the subgroup $\sR^{\lci}_0(X)$ generated by $\nu_*(\divf(f))$, where
$\nu \colon C \to X$ is a good curve and $f \in k(C, \nu^{-1}(X_\sing))^{\times}$.
We let $\CH^F_0(X)$ denote the classical homological Chow group of
0-cycles on $X$ as defined in \cite[Chap.~1]{Fulton}.
Clearly, there are canonical maps $\CH^{LW}_0(X) \surj \CH^{\lci}_0(X) \to \CH^F_0(X)$.

\subsection{The fundamental exact sequence}\label{sec:double-00}
Assume that $X$ is a regular quasi-projective scheme and $D \subset X$ is an effective
Cartier divisor with complement $U$.
To prove Theorem \ref{thm:Main}, we shall use the following fundamental exact
sequence (see \cite[Thm.~1.9]{Binda-Krishna} when $k$ is perfect, \cite[Thm.~2.11]{BKS} if $\dim(X)=2$ and \cite[Thm.~1.1]{GKR}
in the general case).

\begin{thm}\label{thm:Double-ex}
There is a split short exact sequence
  \[
    0 \to \CH_0(X|D) \xrightarrow{p_*} \CH^{\lci}_0(S_X) \xrightarrow{\iota^*} \CH^F_0(X)
    \to 0.
  \]
\end{thm}
In this sequence, $p_*$ takes a 0-cycle on $U$ identically onto $U_+$ and
$\iota^*$ takes a 0-cycle on $U_+\amalg U_-$ onto $U_-$ via projection.

In order to prove our main result, we shall modify slightly the set of relations used to define the Kerz-Saito Chow group of zero-cycles with modulus in the spirit of the Levine-Weibel Chow group.  We proceed as follows.  
We let $\CH^{LW}_0(X|D)$ be the quotient of $\sZ_0(U)$
by the subgroup $\sR^{LW}_0(X|D)$ generated by $\divf(f)$, where
\begin{enumerate}
\item
$C \subset X$ is an integral curve with the property
that $C \not\subset D$;
\item
  $C$ is regular at every point of $E := C \cap D$;
\item
  $f \in \Ker(\sO^{\times}_{C,E} \surj \sO^{\times}_E)$.
\end{enumerate}

Clearly, the difference between $\CH^{LW}_0(X|D)$  and $\CH_0(X|D)$ is in the requirement that the the curves giving the rational equivalence (that we see here as embedded in $X$) are regular in a neighborhood of every point of intersection with the divisor $D$.  By taking normalizations, each such curve gives rise to a curve allowed in the definition of $\sR_0(X|D)$, hence 
there is a canonical surjection
\[\CH^{LW}_0(X|D) \surj \CH_0(X|D).\]
Note also that  the inclusion $\sZ_0(U_+) \inj \sZ_0((S_X)_\reg) = \sZ_0(U_+) \oplus \sZ_0(U_-)$ induces
a push-forward map $p_* \colon \CH^{LW}_0(X|D) \to \CH^{LW}_0(S_X)$ (see the proof of
\cite[Prop.~5.9]{Binda-Krishna}).

\vskip .3cm

\subsection{The factorization lemma}\label{sec:Factor}
We now fix an integral and smooth projective $k$-scheme $X$ of dimension $d \ge 1$.
Let $D \subset X$ be a reduced effective Cartier divisor.
Let $U$ denote the complement of $D$.
The key step in the proof of Theorem \ref{thm:Main} is the following factorization lemma.

\begin{lem}\label{lem:Factor-key}
  Assume that $k$ is infinite and one of the conditions (1), (2) and (4) of
  Theorem \ref{thm:Main} holds.
  Then the (injective) map $\CH_0(X|D) \xrightarrow{p_*} \CH^{\lci}_0(S_X)$ has a
  factorization
  \[
    \CH_0(X|D) \xrightarrow{\phi_{X|D}} H^S_0(U) \xrightarrow{\wt{p}_*} \CH^{\lci}_0(S_X).
    \]
  \end{lem}
\begin{proof}
    We let $C$ be an integral normal curve and let $\nu \colon C \to X$ be a finite morphism
    whose image is not contained in $D$ such that $\nu$ is birational to its image.
    We let $E = \nu^*(D)$ and let
    $f \in \Ker(\sO^{\times}_{C, E} \surj \sO^{\times}_{E_\red})$. 
    We need to show that $p_*(\nu_*(\divf(f))) = 0$. We do it in few steps.
We write $V = C \setminus E$.

\vskip .2cm

    {\bf{Step~1.}}
    We can factorize $\nu$ as
    \begin{equation}\label{eqn:Factor-key-0}
      \xymatrix@C.8pc{
        C \ar[r]^-{\nu'} \ar[dr]_-{\nu} & \P^n_X \ar[d]^-{\pi} \\
        & X}
    \end{equation}
    for some $n \ge 0$ such that $\nu'$ is a regular closed immersion and $\pi$ is the
    canonical projection. We let $X' = \P^n_X$. Then note that
    $S_{X'} \cong \P^n_{S_X}$ by \cite[Prop.~2.3(7)]{Binda-Krishna}. We
    let $D' = \pi^*(D) = \P^n_D$ and $U' = X' \setminus D' = \P^n_U$. Note that $D'$ is
    a reduced divisor on $X'$. Furthermore, if $D$ satisfies any of the conditions
    given in the statement of Theorem \ref{thm:Main}, then so does $D'$. This is a
    consequence of the smoothness of $\pi$.
    Let $\pi' \colon S_{X'} \to S_X$ be the projection map.

Suppose that the image of $\divf(f)$ under the composite map
\begin{equation}\label{eqn:Factor-key-00}
  \sZ_0(V) \xrightarrow{\nu'_*} \sZ_0(U') \surj \CH^{LW}_0(X'|D') \xrightarrow{p_*}
  \CH^{LW}_0(S_{X'})
\end{equation}
is zero.
By composing further with the canonical surjection
$\CH^{LW}_0(S_{X'}) \surj \CH^{\lci}_0(S_{X'})$,
    we see that $\divf(f)$ dies in $\CH^{\lci}_0(S_{X'})$ under $p_*$.
    Let $\pi'_* \colon \CH^{\lci}_0(S_{X'}) \to \CH^{\lci}_0(S_X)$ be the push-forward map,
    which exists by \cite[Prop.~3.18]{Binda-Krishna}.
    It is then clear that
    \[
      p_* \circ \nu_*(\divf(f)) = \pi'_* \circ p_* \circ \nu'_*(\divf(f)) = 0.
    \]
    We thus need to show that if $\nu \colon C \inj X$ is a regular closed immersion, then the image
    of $\divf(f)$ under the composite map
    $\sZ_0(V) \inj \sZ_0(U) \surj \CH^{LW}_0(X|D) \xrightarrow{p_*} \CH^{LW}_0(S_{X})$ is zero. 

\vskip .2cm  

    {\bf{Step~2.}}
    If $X$ is a curve, then $\CH^{LW}_0(X|D) = H^S_0(U)$ and there is nothing to prove.
    We now assume that $X$ is a surface.
    Let $\cyc_{S_X} \colon \CH^{LW}_0(S_X) \to K_0(S_X)$ denote the cycle class map which takes
    the class $[x]$ of a closed point $x \in (S_X)_\reg$ to the class $[\sO_{x}] \in K_0(S_X)$
    (see \cite[Lem.~3.13]{Binda-Krishna}).
    It is shown in \cite[Thm.~7.7]{BKS} (based on the original result due to Levine) that $\cyc_{S_X}$ is injective and its image is
    $SK_0(S_X)$.
    It suffices therefore to show that $\cyc_{S_X} \circ p_*(\divf(f)) = 0$ in $K_0(S_X)$.
    By Proposition~\ref{prop:K-KH} (which can be applied if  one of the conditions (1), (2) or (4) of Theorem \ref{thm:Main} hold), it suffices to show that $\cyc_{S_X} \circ p_*(\nu_*\divf(f))$ dies in
    $KH_0(S_X)$.

We let ${S}_C$ and ${S}^{sn}_C$ be the doubles of $C$ along $E$ and $E_\red$, respectively.
It is then easy to see that the canonical map $\psi \colon {S}^{sn}_C \to S_C$ is the
seminormalization morphism (for example, one can use \cite[Prop.~4.2(1)]{Krishna-torsion} noting that the conductor subscheme of $S^n_C= C\amalg C\to S^{sn}_X$ is reduced and that $S^{sn}_C$ is Cohen-Macaulay, hence $S_2$).  We let $h \in \sO^{\times}_{{S}^{sn}_C, E}$ be the rational
function on ${S}^{sn}_C$ such that $h|_{C_+} = f$ and $h|_{C_-} = 1$.
    Note that the condition $f \in \Ker(\sO^{\times}_{C, E} \surj \sO^{\times}_{E_\red})$ and
    the exact sequence (e.g., see the proof of \cite[Lem.~2.2]{Binda-Krishna})
\[
0 \to \sO^{\times}_{S^{sn}_C, E} \to
    \sO^{\times}_{C, E} \times \sO^{\times}_{C, E} \to
    \sO^{\times}_{E_\red} \to 0
    \]
    imply that $h$ is well defined. Note that $h$ is also a rational function on $S_C$
    but may not lie in $\sO^{\times}_{S_C, E}$.

To simplify the notation, let us write $p_*$ also for the (injective) maps
\[p_*\colon \CH_0(C|E) \to \CH_0^{\rm l.c.i.}(S_C), \quad p_*\colon \CH_0(C|E_{\rm red}) \to
  \CH_0^{\rm l.c.i}(S^{sn}_C) \]
given by the fundamental sequence applied to the pairs $(C,E)$ and $(C, E_{\rm red})$
respectively.  
It follows from the above discussion that 
 $p_*(\divf_{C}(f)) = \divf_{S^{sn}_C}(h) = 0$ in $\CH^{\rm l.c.i}_0({S}^{sn}_C) \cong
 \CH^{LW}_0({S}^{sn}_C)$ (the latter isomorphism holds for any $1$-dimensional reduced
 scheme \cite[Lem. ~3.12]{Binda-Krishna}). By the same token, we get that
 $\cyc_{S^{sn}_C}( p_*(\divf_{C}(f)) )=0$ in $K_0(S^{sn}_C)$, and a fortiori that
 $p_*(\divf(f))$ dies in $KH_0({S}^{sn}_C)$. 
 On the other hand, it easy to see using the excision sequence for $KH$-theory
 \cite[IV.12.6]{Weibel} that
    the canonical map $KH_0(S_C) \xrightarrow{\psi^*} KH_0({S}^{sn}_C)$ is an isomorphism.
    We conclude that the image of $\divf(f)$ under the composite map
    $\sZ_0(C \setminus E) \surj \CH_0(C|E) \to K_0(S_C) \to KH_0(S_C)$ is zero.

Since $\nu' \colon S_C \inj S_X$ is a regular closed immersion (see \S~\ref{sec:Double}),
    there is a push-forward map $\nu'_* \colon KH(S_C) \to KH(S_X)$ (see \S~\ref{sec:KKH}).
 We now consider the commutative diagram
    \[
      \xymatrix{
      \sZ_0(C \setminus E) \ar@{->>}[r] \ar[d] & \CH_0(C|E) \ar[r]^-{\cyc_{S_C} \circ p_*}
      \ar[d]^-{\nu_*} & KH_0(S_C) \ar[d]^-{\nu'_*} \\
      \sZ_0(U) \ar@{->>}[r] & \CH_0(X|D) \ar[r]^-{\cyc_{S_X} \circ p_*} &
      KH_0(S_X).}
    \]
 Using this, we get
  \[
 \cyc_{S_X} \circ p_* \circ \nu_*(\divf(f)) =
    \nu'_* \circ \cyc_{S_C} \circ p_*(\divf(f)) = 0.
  \]    
  This concludes the proof of the lemma when $X$ is a surface (in particular, this covers
  the case where (2) holds).

\vskip .2cm  

{\bf{Step~3.}}
We now assume $d \ge 3$ and fix a closed embedding $X \inj \P^n_k$. Let
$\{D_1, \ldots , D_r\}$ be the set of all irreducible components of $D$.
Let $\{E_1, \ldots , E_s\}$ be the set of irreducible components of $D_\sing$.
We let $\Delta(X) \subset X$ be the set defined in such a way that
$x \in \Delta(X)$ if and only if $x$ is a generic point of one of the schemes
$X, D$ and $D_\sing$.
We assume that we are in  the case (1), namely, $D$ is a normal crossing $k$-scheme.
In particular, each $D_i$ is smooth over $k$ of dimension $d-1$
and each $E_j$ is smooth over $k$ of dimension $d-2 \ge 1$.
Since $C$ is regular and not contained in any of the $D_i$'s, it follows that
the scheme theoretic intersection $C \cap D_i$ is a finite closed subscheme of
$C$. Since the dimension of the closure of each of the points of $\Delta(X)$ is
at least one, it follows that $\Delta(X) \cap C = \emptyset$. 

Given the above arrangement of $X, D_i, E_j$ and $C$ in $\P^n_k$, we can apply either the
Bertini theorem of Altman-Kleiman \cite[Thm.~7]{Altman-Kleiman} or of
Ghosh-Krishna \cite[Thm.~3.9]{Ghosh-Krishna-Bertini}) to find a complete intersection
hypersurface $H = H_1 \cap \cdots \cap H_{d-2}$ in $\P^n_k$ of large enough degrees
containing $C$ such that $Y = X \cap H$ satisfies the following.

\begin{enumerate}
\item
  $Y$ is a smooth $k$-scheme of pure dimension two.
\item
  Each $Y \cap D_i$ is a smooth $k$-scheme of dimension one.
\item
  Each $Y \cap E_j$ is a smooth $k$-scheme of dimension zero.
\item
  $Y \cap D_J = Y \cap ({\underset{i \in J}\bigcap} D_i) = \emptyset$ if $|J| \ge 3$.
\end{enumerate}

Since $\dim(D_i) \ge 2$, it follows (for instance, from \cite[Cor.~6.2]{Hartshorne-ample})
that $Y$ and well as each $Y \cap D_i$ is connected, hence integral.
Furthermore, $Y \cap D$ is a curve which is reduced away from $C$ by
\cite[Thm.~3.2]{Ghosh-Krishna-Bertini}.
Since $C \cap D$ is finite, $Y$ is regular and $Y \cap D$ is a Cartier divisor on $Y$, it
follows that $Y \cap D$ is a Cohen-Macaulay curve which is generically reduced. This
implies that $Y \cap D$ must be reduced (see \cite[Prop.~14.124]{GW}).
We let $F = Y \cap D$. Then we conclude from (1), (2) and (3) above that
$Y$ is a complete intersection smooth integral surface inside $X$ which contains
$C$ and $F = Y \cap D$ is a normal crossing curve on $Y$.

If condition (4) holds, i.e., if $k \subset \ov{\Q}$, then we can repeat the above
argument to find a 
complete intersection smooth integral surface $Y \subset X$ which contains $C$
and $F = Y \cap D$ is a reduced Cartier divisor on $Y$. The only difference is that we
can not no longer guarantee that the irreducible components of $F$ are regular. 

In any case, let $(Y, F)$ be the pair constructed above, and let $\tau \colon Y \inj X$ be
the inclusion map.  Let $W = Y \setminus F = U \cap Y$.
Let $\tau' \colon S_Y \inj S_X$ denote the inclusion map, where $S_Y$ is the double of $Y$
along $F$. Then $\tau'$ is a regular closed embedding by \cite[Prop.~2.4]{Binda-Krishna}
(see \S~\ref{sec:Double}).
It follows from Step~2 that the image of $\divf(f)$ under the composite map
$\sZ_0(V) \inj \sZ_0(W) \surj \CH^{LW}_0(Y|F) \xrightarrow{p_*} \CH^{LW}_0(S_{Y})$ is zero.
We now consider the commutative diagram
\begin{equation}\label{eqn:Factor-key-1}
      \xymatrix{
        \sZ_0(W) \ar@{->>}[r] \ar[d] & \CH^{LW}_0(Y|F) \ar[r]^-{p_*} \ar[d]^-{\tau_*} &
        \CH^{LW}_0(S_{Y}) \ar[d]^-{\tau'_*} \\
        \sZ_0(U)   \ar@{->>}[r] & \CH^{LW}_0(X|D) \ar[r]^-{p_*} & \CH^{LW}_0(S_{X}).}
    \end{equation}

We note that $\tau'_*$ exists because $\tau'$ is a regular closed immersion and
    $S_Y \cap (S_X)_\reg = (S_Y)_\reg$. One easily checks that the push-forward map
    $\tau_*$ also exists and the diagram commutes.
    We thus get
    \[
p_*(\divf(f)) = p_* \circ \tau_*(\divf(f)) =
    \tau'_* \circ p_*(\divf(f)) = 0.
  \]          
This concludes the proof of the lemma.  
\end{proof}

\vskip .3cm

To take care of the case when $k$ is a finite field, we shall need the following
 result. For any $X \in \Sch_k$ and ${k'}/k$ a field extension, we let
$X_{k'} = X \times_{\Spec(k)} \Spec(k')$ with projection $v \colon X_{k'} \to X$.

\begin{lem}[\cite{Schmidt-ant}, p.191]\label{lem:Suslin-hom-inv*}
  Let $X$ be a smooth quasi-projective $k$-scheme and let ${k'}/k$ be an algebraic
  field extension. Then the flat pull-back on 0-cycles induces a homomorphism
  \[
    v^* \colon H^S_0(X) \to H^S_0(X_{k'}).
  \]
  If ${k'}/k$ is finite, then the push-forward on 0-cycles induces a
  homomorphism
  \[
 v_* \colon  H^S_0(X_{k'}) \to H^S_0(X)
  \]
  such that $v_* \circ v^*$ is multiplication by $[k': k]$.
\end{lem}

We remark that the pull-back map $v^*$ is defined in \cite{Schmidt-ant} for finite
field extensions. But this implies the case of arbitrary algebraic extensions by an
easy limit argument.

\vskip .3cm

\subsection{Proof of Theorem \ref{thm:Main}}\label{sec:Pf-main*}
We let $X,D$ and $U$ be as in Theorem \ref{thm:Main}.
 We first assume that $k$ is infinite and one of the conditions (1), (2) and (4) of
 Theorem \ref{thm:Main} holds. In this case, the map
 $p_* \colon \CH_0(X|D) \to \CH^{\lci}_0(S_X)$ is injective by
 Theorem \ref{thm:Double-ex}.
 %\cite[Thm.~5.10]{Binda-Krishna} (when $k$ is perfect) and \cite[Prop.~4.4]{GKR} (when $k$ is arbitrary).
 Combining this injectivity with Lemma \ref{lem:Factor-key},
one immediately concludes that $\phi_{X|D}$ must be an isomorphism.

We now assume that $k$ is finite and one of the conditions (1), (2) and (4) of
 Theorem \ref{thm:Main} holds.
We only have to show that $\phi_{X|D}$ is injective.
Let $\alpha \in \CH_0(X|D)$ be a class such that $\phi_{X|D}(\alpha) = 0$.
Let $\ell_1 \neq \ell_2$ be two primes different from $\Char(k)$. Let ${k_i}/k$ be the
pro-$\ell_i$
field extension of $k$ for $i = 1,2$.

Using \cite[Prop.~8.5]{Gupta-Krishna-BF}, we have a commutative diagram
\[
  \xymatrix@C.8pc{
    \CH_0(X|D) \ar[r] \ar[d] & H^S_0(U) \ar[d] \\
    \CH_0(X_{k_i}|D_{k_i}) \ar[r] & H^S_0(U_{k_i}),}
\]
where the vertical arrows are the base change maps. The right vertical arrow exists
by Lemma \ref{lem:Suslin-hom-inv*}. Using the case of infinite fields, it
follows that $\alpha$ dies in $\CH_0(X_{k_i}|D_{k_i})$. In particular, it dies in
$\CH_0(X_{k'_i}|D_{k'_i})$ for a finite extension $k'_i$ whose degree is a power of
$\ell_i$ for
each $i = 1,2$. Using the projection formula for Chow groups with modulus
(see \cite[Prop.~8.5]{Gupta-Krishna-BF}),
we conclude that $\ell^{n_1}_1 \alpha = \ell^{n_2}_2 \alpha = 0$ in $ \CH_0(X|D)$ for some
$n_1, n_2 \ge 1$. It follows that $\alpha = 0$. This concludes the proof of
Theorem \ref{thm:Main} under the conditions (1), (2) and (4).
The remaining case (3) is already shown in \cite[Thm.~1.1]{Binda-Krishna-21}.
This concludes the proof of Theorem \ref{thm:Main}.
\qed

\vskip .3cm

\subsection{Proofs of some applications}\label{sec:LW-pf}
In this subsection, we shall give the proofs of the some of the
applications of Theorem \ref{thm:Main}
mentioned in \S~\ref{sec:Intro}. Corollary~\ref{cor:Russell} and 
Theorem~\ref{thm:Main-0-0}
are immediate from Theorem \ref{thm:Main} using the references given before their
statements. We shall therefore prove Theorem \ref{thm:RTT}.

\vskip .3cm

{\bf {Proof of Theorem \ref{thm:RTT}:}}
Since the theorem is already known for torsion away from the
characteristic by \cite{SS}, we shall assume that $k$ is algebraically closed of positive characteristic. We consider the diagram
  \begin{equation}\label{eqn:doubl-ex-seq-*-0}
\xymatrix@C.8pc{
\CH_0(X|D)_0 \ar[r]^-{\phi_{X|D}} \ar[d]_{\alb_{X|D}} &
H^{S}_0(U)_{0} \ar[d]^{\alb_U} \\
J^d(X|D)(k) \ar[r]^-{\lambda_{X|D}} & \Alb(U)(k),}
\end{equation}
where $J^d(X|D)$ is the semi-abelian Albanese variety with modulus, constructed in
\cite[\S~11.1]{Binda-Krishna}.
This diagram is commutative and the bottom horizontal arrow is an isomorphism by
\cite[Thm.~3.2]{Binda-Krishna-21}.
The left vertical arrow is an isomorphism on the torsion subgroups
by \cite[Thm.~6.7]{Krishna-torsion}.
The top horizontal arrow is an isomorphism by Theorem \ref{thm:Main}. 
The desired assertion follows.
\qed

\vskip .3cm

\section{Motivic cohomology of normal crossing schemes}
\label{sec:NCS}
The goal of this section is to prove Theorem \ref{thm:LW-MC}.
We shall need few ingredients in order to achieve this. The first is a perfection
property of the cycle groups which we recall below.

\subsection{Perfection property of cycle groups}\label{sec:App-pf}
We let $k$ be a field of exponential characteristic $p$. We let $\Lambda$ be a commutative
ring which we assume to be $\Z$ if ${\Char}(k) = 0$ or any $\Z[\tfrac{1}{p}]$-algebra if
$\Char(k) = p > 0$. 
We begin with a short recap about motivic cohomology of $k$-schemes, and related motivic
invariants.
Recall our notation that for a field extension ${k'}/k$ and $X \in \Sch_k$,
we write $X_{k'}$ for the base
change of $X$ by $k'$ over $k$ and $v \colon X_{k'} \to X$ denotes the projection map.

\subsection{Motivic homology and cohomology of singular schemes}\label{sec:MC*}
Let $X\in \Sch_k$ with the structure map $f\colon X \to  \Spec(k)$ and let
$m, n \in \Z$.

\begin{defn}\label{defn:MC-sing}
  The motivic cohomology groups of $X$ are defined as
  \[
    H^m(X, \Lambda(n)) = \Hom_{\mathbf{DM}(k, \Lambda)}(M(X), \Lambda(n)[m]),
\]
where $\mathbf{DM}(k, \Lambda)$ is Voevodsky's non-effective category of motives for the
cdh-topology (also known as the `big' category of motives) with $\Lambda$-coefficients,
$\Lambda(n)$ is the motivic complex, and
$M(X)$ is the motive of $X$ (see \cite{Suslin-Voevodsky} or
\cite[\S~1]{Cisinski-Deglise}).
\end{defn}

Let $\sS\sH(X)$ be the monoidal stable homotopy
category of smooth schemes over $X$ and $\sS\sH_{\cdh}(k)$ the stable homotopy
category of $\Sch_k$ with respect to the cdh topology
(e.g., see \cite[\S~2]{Krishna-Pelaez}). There is an adjoint pair
of functors $(\psi_X, \phi_X) \colon \sS\sH(X) \to \mathbf{DM}(k, \Lambda)$.
By \cite[Thm.~5.1]{Cisinski-Deglise} and \cite[Thm.~2.14]{Krishna-Pelaez},
these functors give rise to functorial isomorphisms
\begin{equation}\label{eqn:MV-sing-0} 
\begin{array}{lll}
    H^p(X, \Lambda(q))   &  \xrightarrow{\,\,\cong\,\,}& \Hom_{\sS\sH(X)}(\mathbb{S}_X,  \Sigma^{p,q} (H\Lambda_X)) \\
                      & \xrightarrow{\,\,\cong\,\,} & \Hom_{\sS\sH_{\cdh}(k)}(\Sigma^\infty_T X_+, \Sigma^{p,q}H\Lambda),
  \end{array}
\end{equation}
where $\mathbb{S}_X$ is the sphere spectrum (the unit object) of $\sS\sH(X)$,  $H\Lambda$ is the motivic Eilenberg-MacLane
spectrum in $\sS\sH(k)$, and $H\Lambda_X = {\bf L}f^*(H\Lambda)$ for the structure map
$f \colon X \to \Spec(k)$.
 We refer to, e.g., \cite[\S~2,3]{Krishna-Pelaez}
 for the definitions of the suspension operators
 $\Sigma^\infty_T$ and $\Sigma^{p,q}$.

In a similar fashion, one can define motivic cohomology groups with compact support and
motivic homology as follows.
\[
  \begin{array}{lll}
  H^{m}_c(X, \Lambda(n)) & = & \Hom_{\mathbf{DM}(k, \Lambda)}(M_c(X), \Lambda(n)[m]) \\
  H_{m}(X, \Lambda(n)) & = & \Hom_{\mathbf{DM}(k, \Lambda)}(\Lambda(n)[m], M(X)),
  \end{array}
  \]
where $M_c(X)$ is the motive of $X$ with compact support \cite[Defn.~16.13]{MVW}. 
In particular, there is a canonical isomorphism \cite[Prop.~14.18]{MVW}: 
\[H_n(X, \Lambda(0)) \xrightarrow{\simeq} H_n^{S}(X)_{\Lambda},\]
where the right-hand side is the $n$-th Suslin homology group of $X$  recalled in
\ref{sec:CHMS}. 

We recall the following result of Cisinski-D{\'e}glise
\cite[Prop.~8.1]{Cisinski-Deglise}.

\begin{thm}\label{lem:EK-0}
  Let ${k'}/k$ be a purely inseparable field extension and let
  $v \colon \Spec(k') \to \Spec(k)$ be the projection map. Then 
the pull-back functor
 \[
u^* \colon \mathbf{DM}(k, \Lambda) \to  \mathbf{DM}(k', \Lambda)
\]
is an equivalence of triangulated categories.
\end{thm}

Using Theorem~\ref{lem:EK-0} and the description of various groups above, we get the
following result which we shall use in our proofs.

\begin{cor}\label{cor:EK}
 Let ${k'}/k$ be a purely inseparable field extension and let
  $u \colon \Spec(k') \to \Spec(k)$ be the projection map. 
Then for any $X \in \Sch_k$, the pull-back maps
\begin{align*}
 v^*\colon H^m(X, \Lambda(n)) &\longrightarrow H^m(X_{k'}, \Lambda(n))\\ 
 v^*\colon H^m_c(X, \Lambda(n)) &\longrightarrow H^m_c(X_{k'}, \Lambda(n))
 %v_*\colon H_{m}(X_{k'}, \Lambda(n)) &\longrightarrow  H_{m}(X, \Lambda(n))    
  \end{align*}
are isomorphisms. 
\end{cor}

Assume now that $X$ is smooth of pure dimension $d$ over $k$. 
Duality in motivic homotopy theory makes it possible to identify motivic cohomology and
homology groups (as well as their compactly  supported version) with the appropriate twist
and shift. We shall need an explicit description, in the bi-degree $(2d,d)$, of the map
realizing such duality isomorphism for later applications. We quickly recall its
construction. For every closed point $x\in X$, the inclusion
$\Spec(k(x)) \inj X$ gives  a Gysin map
$ M_c(X) \to M_{\{x\}}(X)  \xrightarrow{\cong} M(k(x))(d)[2d]$. 
Taking cohomology, we get \begin{equation}\label{eqn:EKW-inv-3}
  \Z \xrightarrow{\simeq} H^0(k(x), \Z(0)) \to  H^{2d}_c(X, \Z(d)),
\end{equation}
and extending \eqref{eqn:EKW-inv-3} by linearity, we get 
$\phi_X \colon \sZ_0(X) \to H^{2d}_c(X, \Z(d))$.

\begin{lem}\label{lem:SH-MCCS}
  The map $\phi_X$ descends to an isomorphism
  \[
    \phi_X \colon H^S_0(X)_\Lambda \xrightarrow{\cong} H^{2d}_c(X, \Lambda(d)).
  \]
\end{lem}
\begin{proof}
  Let $k'$ be a perfect closure of $k$
and consider the commutative diagram
  \begin{equation}\label{eqn:SH-MCCS-0}
    \xymatrix@C.8pc{
      H^S_0(X)_\Lambda \ar[r]^-{\phi_X} \ar[d]_-{v^*} &  H^{2d}_c(X, \Lambda(d))
      \ar[d]^-{v^*} \\
      H^S_0(X_{k'})_\Lambda \ar[r]^-{\phi_{X_{k'}}} & H^{2d}_c(X_{k'}, \Lambda(d)).}
  \end{equation}
  
  The left vertical arrow is an isomorphism by Lemma~\ref{lem:Suslin-hom-inv*} using
  a limit argument and
  the right vertical arrow is an isomorphism by Corollary~ \ref{cor:EK}.
  The bottom horizontal arrow is an isomorphism (e.g., by
  \cite[Thm.~5.5.14]{Kelly}).
  The lemma now follows.
 \end{proof}

 \begin{comment}
  Finally, we shall need the following result in the proof of Theorem \ref{thm:LW-MC}.

\begin{lem}\label{lem:PBF}
  The cohomology theory $X \mapsto {\underset{m,n \in \Z}\bigoplus} H^m(X, \Lambda(n))$
  on $\Sch_k$ satisfies the projective bundle formula.
\end{lem}
\begin{proof}
By Corollary~\ref{cor:EK}, we can replace $k$ by its perfect closure.
  In particular, we can assume that $k$ is perfect. In the latter case, the lemma
  is \cite[5.5.10]{Kelly}.
\end{proof}
\end{comment}

\vskip .3cm

\subsection{The snc subcurves}\label{sec:snc-curve**}
We fix a normal crossing $k$-scheme
$X$ of dimension $d \ge 1$ and let $\{X_1, \ldots , X_n\}$ be the set of irreducible
components of $X$. A snc subcurve $C \subset X$
(see \cite[\S~2.1]{EKW}) is
a reduced closed subscheme of pure dimension one such that the scheme theoretic
intersection of $C$ with each irreducible component $X_i$ of $X$ is either empty or smooth of pure
dimension one, its intersections with $X_i \cap X_j$ (for all $i \neq j$) are either
empty or smooth and
0-dimensional, and its intersections with
$X_i \cap X_j \cap X_l$ (for all $i \neq j \neq l \neq i$) are empty. 

\begin{remk}\label{remk:Smooth-red}
  We note that the above definition of snc subcurves is more restrictive than the one given
  in \cite[\S~2.1]{EKW} because the latter only requires the
  intersections $C \cap X_i \cap X_j$ (for $i \neq j$)
  to be reduced (not necessarily smooth) and 0-dimensional.
  The stronger assumption allows us to prove the following result. But this distinction
  disappears if $k$ is perfect.
\end{remk}

\begin{lem}\label{lem:SNC-PB}
  Let $X \in \Sch_k$ be a normal crossing scheme and let $C \subset X$ be
  a snc subcurve. Let ${k'}/k$ be a finite purely inseparable field extension.
  Then $X_{k'}$ is a normal crossing $k'$-scheme and $C_{k'} \subset X_{k'}$ is
  a snc subcurve.
\end{lem}
\begin{proof}
  Let $v \colon X_{k'} \to X$ be the base change morphism. Then $v$ is a
  universal homeomorphism. In particular, there is a bijective correspondence between
  the irreducible components of $X$ and $X_{k'}$. We let $X'_i = (X_i)_{k'}$ for
  $1 \le i \le n$. Since $X_i \in \Sm_k$, it follows each
  $X'_i$ is integral and smooth over $k'$. In turn, this implies that
  $X_{k'}$ is generically reduced (i.e., $X_{k'}$ satisfies Serre's $R_0$-condition).
  Since $v$ is finite and flat, and $X$ satisfies
  Serre's $S_1$-condition (because it is reduced), it follows that $X_{k'}$ also satisfies
  Serre's $S_1$-condition. It follows that $X_{k'}$ is reduced. By the same token,
  for every nonempty subset $J \subset [1, n]$, the scheme theoretic 
intersection $X'_J := {\underset{i \in J}\bigcap} X'_i$ is a smooth $k'$-scheme
(unless empty) of pure dimension $d + 1 - |J|$. In other words, $X_{k'}$ is a normal
crossing
$k'$-scheme. An identical proof shows that $C_{k'} \subset X_{k'}$ is a snc subcurve.
\end{proof}

\vskip .3cm
\subsection{The cycle group $\CH^{EKW}_0(X)$}\label{sec:EKW**}
Let $X \in \Sch_k$ be a normal crossing scheme as above.
The cycle group $\CH^{EKW}_0(X)$ is the quotient of $\sZ_0(X_\reg)$
by the subgroup $\sR^{EKW}_0(X)$ generated
by $\divf(f)$, where $f \in k(C)^{\times}$ is a rational function on a curve 
$C \subset X$ such that the pair $(C,f)$ satisfies either of the conditions 
(1) and (2) below.
\begin{enumerate}
\item
$C$ is an integral curve not contained in $X_\sing$ with normalization
$\nu: C^n \to C \inj X$ and $f \in \sO^{\times}_{C^n, \nu^*(X_\sing)}$ such that 
$f(x) = 1$ for all $x \in \nu^*(X_\sing)$.
\item
$C \subset X$ is a snc subcurve and $f \in \sO^{\times}_{C, (C\cap X_\sing)}$.
\end{enumerate}

Let $Y = X_\sing$ for the normal crossing scheme $X$.
The inclusion $\Spec(k(x)) \inj Y_{\reg}$ gives a Gysin homomorphism
$k(x)^{\times} \xrightarrow{\cong} H^1(k(x), \Z(1)) \to H^{2d-1}(Y, \Z(d))$
for every closed point $x \in Y_\reg$. Note that $Y_\reg \in \Sm_k$ since $X$ is
a normal crossing $k$-scheme. Hence, we get the global Gysin map
${\underset{x \in Y^{(d-1)}_\reg}\bigoplus} k(x)^{\times} \to H^{2d-1}(Y, \Z(d))$,
where $Y^{(d-1)}_\reg$ is the set of closed points of $Y_\reg$.

\begin{lem}\label{lem:MC-sing}
  The map
  \[
    \alpha_Y \colon
    {\underset{x \in Y^{(d-1)}_\reg}\bigoplus} k(x)^{\times} \otimes_{\Z} \Lambda \to
    H^{2d-1}(Y, \Lambda(d))   
  \]
  is surjective.
\end{lem}
\begin{proof} This is \cite[Prop.~6.4]{EKW} if $k$ is perfect. The general case follows from the perfect one using Corollary~\ref{cor:EK} in order to
  identify $H^{2d-1}(Y, \Lambda(d))$ with $H^{2d-1}(Y_{k'}, \Lambda(d))$, where ${k'}$  is a
  perfect closure of $k$. 
\end{proof}

Let us now assume that $X$ is a projective normal crossing $k$-scheme.
The next step for  proving Theorem \ref{thm:LW-MC} is the description of the motivic cohomology groups $H^{2d}(X, \Lambda(d))$ for normal crossing varieties discussed in \cite{EKW}. %  for  following extension of \cite[Thm.~7.1]{EKW} to normal crossing schemes over imperfect fields.
We have seen in ~\eqref{eqn:EKW-inv-3} that there is a canonical map
$\phi_{X_\reg} \colon \sZ_0(X_\reg) \to H^{2d}_c(X_\reg, \Z(d))$. Composing with the
map $ H^{2d}_c(X_\reg, \Z(d)) \to  H^{2d}(X, \Z(d))$, we get
$\wt{\lambda}_X \colon \sZ_0(X_\reg) \to H^{2d}(X, \Z(d))$.

\begin{prop}\label{prop:EKW-inv}
  The map $\wt{\lambda}_X$ induces an isomorphism
\[
  \wt{\lambda}_X \colon \CH^{EKW}_0(X)_\Lambda \xrightarrow{\cong} H^{2d}(X, \Lambda(d)).
\]
\end{prop}
\begin{proof} This is \cite[Thm.~7.1]{EKW} when $k$ is perfect. The same proof works in
  the general case,  using Lemma \ref{lem:MC-sing} instead of \cite[Prop.~6.4]{EKW}
  and passing to a perfect closure of $k$. The latter is achieved using
  Corollary~\ref{cor:EK} and Lemma \ref{lem:SNC-PB}.
  The vanishing $H^{2d}(Y, \Lambda(d)) = 0$ for $Y=X_{\rm sing}$, that
  is also used in the proof of \cite[Thm.~7.1]{EKW}, can be deduced from
  \cite[Thm.~5.1]{Krishna-Pelaez} using again Corollary~\ref{cor:EK}. 
\end{proof}

\vskip .3cm

\subsection{Proof of Theorem \ref{thm:LW-MC}(1)}\label{sec:pf-1}
Let $k$ be any field and let $X \in \Sch_k$ be as in part (1) of
Theorem \ref{thm:LW-MC}. In other words, $X$ is a reduced quasi-projective $k$-scheme of
pure dimension $d$.
To construct the map $\lambda_X \colon \CH^{\lci}_0(X) \to H^{2d}(X, \Z(d))$,
we proceed as follows.

Using ~\eqref{eqn:MV-sing-0} and \cite[Defn.~2.30, Thm.~2.31]{Navarro}, we have
a Gysin map $\tau_x \colon \Z \cong H^0(k(x), \Z(0)) \to H^{2d}(X, \Z(d))$
for any closed point $x \in X_\reg$.
Extending this linearly, we get a homomorphism
$\lambda_X \colon \sZ_0(X_\reg) \to H^{2d}(X, \Z(d))$.

When $d =1$, it is shown in the proof of \cite[Lem.~7.12]{Krishna-Pelaez} that
$\lambda_X$ factors through the Chow group (this uses the slice spectral sequence
for singular schemes). For $d \ge 2$,
we let $\nu \colon C \to X$ be a good curve and let $f \in \sO^\times_{C,S}$,
where $S = \nu^{-1}(X_\sing) \cup C_\sing$. By \cite[Lem.~3.4]{Binda-Krishna}, we can assume
that $\nu$ is a lci morphism. In particular, there is a Gysin homomorphism
$\nu_* \colon H^2(C, \Z(1)) \to H^{2d}(X, \Z(d))$ by \cite[Defn.~2.31, Thm.~2.31]{Navarro}.
We now consider the diagram
\begin{equation}\label{eqn:CH-MC}
  \xymatrix@C1pc{
    \sZ_0(C \setminus S) \ar[r]^-{\lambda_C} \ar[d]_-{\nu_*} &
    H^2(C, \Z(1)) \ar[d]^-{\nu_*} \\
    \sZ_0(X_\reg) \ar[r]^-{\lambda_X} & H^{2d}(X, \Z(d)).}
\end{equation}

It is immediate from the construction of $\lambda_X$ and Gysin maps that
this diagram is commutative.
By the curve case, we have that $\lambda_C(\divf(f)) = 0$.
It follows that
\[
\lambda_X(\divf(f)) = \lambda_X \circ \nu_*(\divf(f)) =
\nu_* \circ \lambda_C(\divf(f)) = 0.
\]
This shows that $\lambda_X$ factors through a homomorphism
$\lambda_X \colon \CH^{\lci}_0(X) \to   H^{2d}(X, \Z(d))$.

\begin{comment}
In other words, $X$ is a reduced quasi-projective $k$-scheme of
pure dimension $d$.
The functor $X \mapsto  {\underset{m,n \in \Z}\bigoplus}
H^m(X, \Lambda(n))$ on $\Sch_k$ satisfies the projective bundle formula by
Lemma \ref{lem:PBF}. All other axioms of Gillet \cite{Gillet} for admitting Chern
classes from higher $K$-theory are well known to be satisfied for this functor.
It follows from Gillet's theory of Chern classes that there are
natural transformations of contravariant functors
\begin{equation}\label{eqn:Chern}
  c_{n,m} \colon K_m(X) \to H^{2n-m}(X, \Lambda(n))
  \end{equation}
  on $\Sch_k$ for all $m, n \in \Z$, where $K_*(X)$ is Quillen's algebraic $K$-theory.

By \cite[Lem.~3.13]{Binda-Krishna}, there exists a cycle class map
$\cyc_{X} \colon \CH^{\lci}_0(X) \to K_0(X)$ which takes a closed point $x \in X_\reg$
to the class $[\sO_{x}] \in K_0(X)$. Composing the cycle class map with
the Chern class map $c_{d,0} \colon K_0(X) \to H^{2d}(X, \Lambda(d))$ from
~\eqref{eqn:Chern}, we get our desired homomorphism
\begin{equation}\label{eqn:LW-MC-0}
  \lambda_X \colon \CH^{\lci}_0(X) \to H^{2d}(X, \Lambda(d)).
\end{equation}
\end{comment}

By construction, for a regular closed immersion $f \colon X' \inj X$ of
equidimensional schemes such that
$\dim(X') = d'$ and $f^{-1}(X_\sing) \subset X'_\sing$, there is a commutative diagram
\begin{equation}\label{eqn:LW-MC-1}
  \xymatrix@C.8pc{
    \CH^{\lci}_0(X') \ar[r]^-{\lambda_{X'}} \ar[d]_-{f_*} & H^{2d'}(X', \Z(d'))
    \ar[d]^-{f_*} \\
    \CH^{\lci}_0(X) \ar[r]^-{\lambda_X} & H^{2d}(X, \Z(d)),}
\end{equation}
in which the left and the right vertical arrows are the Gysin homomorphisms of
\cite[Prop.~3.18]{Binda-Krishna} and \cite[Defn.~2.30, Thm.~2.31]{Navarro}, respectively.

\begin{comment}
in which the left and the right vertical arrows are the Gysin homomorphisms of
\cite[Prop.~3.18]{Binda-Krishna} and \cite[Defn.~2.32]{Navarro}, respectively.
The commutativity of the Chern class map with the Gysin map is part of Gillet's
construction of the Chern classes (see \cite[\S~2]{Gillet}).
The cycle class map commutes with the Gysin map by \cite[Prop.~3.18]{Binda-Krishna}.
It follows that the above diagram is commutative.
\end{comment}

\vskip .3cm

\subsection{A key lemma}\label{lem:lci-mod}
We shall need the following key result about the lci Chow group of normal crossing
schemes. Let $X$ be a normal crossing $k$-scheme of dimension $d$ as above and let $Y$
be an irreducible component of $X$. We let $Z \subset X$ be the scheme theoretic
closure of $X \setminus Y$ and let $E = Y \cap Z$.
Then $X_\sing$ and $E$ are normal crossing $k$-schemes of dimension $d-1$ and $E$ is a simple
normal crossing divisor on $Y$. We let $V = Y \setminus E$ so that there is an inclusion
of the 0-cycle groups $\iota_* \colon \sZ_0(V) \inj \sZ_0(X_\reg)$, where
$\iota \colon Y \inj X$ is the inclusion.

\begin{lem}\label{lem:CH-mod-lci}
  Assume that $k$ is infinite. Then the map $\iota_*$ descends to a homomorphism
  \[
    \iota_* \colon \CH_0(Y|E) \to \CH^{\lci}_0(X).
  \]
\end{lem}
\begin{proof}
This is shown in \cite[Thm.~8.3]{Binda-Krishna-21}
when either $k$ is algebraically closed or $d \le 2$. We shall closely
  follow that proof. We can assume that $d \ge 3$ and that the lemma holds in smaller
  dimensions. We fix a locally closed embedding $X \inj \P^N_k$.
  Let $C \subset Y$ be an integral curve not contained in $E$
and let $f \in {\rm Ker}(\sO^{\times}_{C^n,\nu^*(E)} \to \sO^{\times}_{\nu^*(E)})$,
where $\nu: C^n \to Y \inj X$ is the canonical map from the normalization of $C$.
We need to show that $\divf(f)$ dies in $\CH^{\lci}_0(X)$.

We can write $\nu$ as the
composition of two maps $C^n \xrightarrow{\nu'} \P^m_X \xrightarrow{\pi} X$ for some
integer $m \ge 0$, where $\nu'$ is a regular closed immersion and $\pi$ is the projection.
Note that $\nu'$ factors through $\P^m_Y$. We now note that $\P^m_X$ is a normal crossing
$k$-scheme of dimension $d \ge 3$ and $\P^m_Y$ is an irreducible component of $\P^m_X$.
Using the push-forward map $\pi_* \colon \CH^{\lci}_0(\P^m_X) \to \CH^{\lci}_0(X)$
(see \cite[Prop.~3.18]{Binda-Krishna}) and the canonical map
$\CH^{LW}_0(\P^m_X) \to \CH^{\lci}_0(\P^m_X)$, it suffices to show that
$\nu_*(\divf(f))$ dies in $\CH^{LW}_0(\P^m_X)$. We can therefore replace $\CH^{\lci}_0(X)$
with $\CH^{LW}_0(X)$ and assume that $C$ is normal.
Note that the map $\CH^{LW}_0(X) \to \CH^{\lci}_0(X)$ is an isomorphism for $d \le 2$
by \cite[Thm.~8.1]{BKS} (see also \cite{Binda-Krishna}).  Hence, the base case of the
induction holds for the modified problem too.

We can now repeat the argument of the proof of \cite[Thm.~8.3]{Binda-Krishna-21}
(without using any blow-up) to find a hypersurface
section $X' \subset X$ inside $\P^N_k$ containing $C$ such that $X'$ is a
$(d-1)$-dimensional normal crossing $k$-scheme, $X'_\reg = X_\reg \cap H$
and $Y' = X' \cap Y = H \cap Y$ is a smooth irreducible component of $X'$.
It follows by induction that $\nu_*(\divf(f))$ dies in $\CH^{LW}_0(X')$.
In particular, it dies in $\CH^{LW}_0(X)$ via the push-forward map
$\CH^{LW}_0(X') \to \CH^{LW}_0(X)$, induced by the regular closed immersion
$X' \inj X$. This concludes the proof.
\end{proof}

\subsection{Proof of Theorem \ref{thm:LW-MC}(2,3)}\label{sec:pf-2}
If $X$ is projective and $X_\reg$ is smooth over $k$, then
the surjectivity of $\lambda_X$ (as asserted in part (2)) follows from the surjection
$H^{2d}_c(U, \Lambda(d)) \surj H^{2d}(X, \Lambda(d))$ and
Lemma \ref{lem:SH-MCCS}. 
We now prove the last part of the theorem.
We are given that $X$ is a projective normal crossing $k$-scheme and need to
show that $\lambda_X$ is an isomorphism. To prove this, we first assume that $k$ is
infinite and look at the diagram
\begin{equation}\label{eqn:LW-MC-00}
  \xymatrix@C.8pc{
    \sZ_0(U)   \ar@{->>}[r] \ar@{->>}[dr] &
    \CH^{EKW}_0(X)_\Lambda \ar@{.>}[d]^-{\psi_X} \ar[dr]^-{\wt{\lambda}_X} & \\
    & \CH^{\lci}_0(X)_\Lambda \ar[r]^-{\lambda_X} & H^{2d}(X, \Lambda(d)).}
\end{equation}
It suffices to show that $\psi_X$ exists such that the resulting
left triangle commutes. This is shown in the proof of \cite[Thm.~8.4]{Binda-Krishna-21}.
But we do not need $\Lambda$-coefficient for constructing $\psi_X$, thanks to
Theorem \ref{thm:Main}.
We sketch the steps. We let $Y = X_\sing$.

We let $C \subset X$ be a reduced curve and $f \in k(C)^{\times}$ a rational function, where $k(C)$ is the ring of total quotients for $C$.
 We now observe that if the pair $(C,f)$ is of type (1) in the definition of
 $\CH^{EKW}_0(X)$ in \S~\ref{sec:EKW**}, then $C$ must be integral. In particular,
 it must be contained in
one and only one irreducible component $X'$ of $X$. Moreover, for
this component $X'$, the intersection $E = Z\cap Y$ must be a simple normal
crossing divisor on $X'$, where $Z$ is the scheme theoretic closure of $X \setminus X'$.
We now conclude from Theorem \ref{thm:Main} that $\divf(f)$ dies in
$\CH_0(X'|E)$. Hence, it dies in $\CH^{\lci}_0(X)$ by Lemma \ref{lem:CH-mod-lci}.
If $(C,f)$ is of type (2) in the definition of $\CH^{EKW}_0(X)$ in \S~\ref{sec:EKW**}, then
$C \subset X$ is a Cartier curve (see \S~\ref{sec:LWC}) by \cite[Lem.~7.6]{Binda-Krishna-21}
and hence, $\divf(f)$ already dies in $\CH^{LW}_0(X)$.  
This concludes the proof of part (3) when $k$ is infinite.

We now assume that $k$ is finite. Since we already showed surjectivity of
$\lambda_X$ above, we only need to show that
it is injective. We let $k'$ be the pro-$p$-extension of $k$ where
$\Char(k) = p$. We then get a commutative diagram
\begin{equation}\label{eqn:finite-case}
  \xymatrix@C.8pc{
\CH^{\lci}_0(X)_\Lambda \ar[r]^-{\lambda_X} \ar[d]_-{v^*} & H^{2d}(X, \Lambda(d)) \ar[d]^-{v^*} 
\\
\CH^{\lci}_0(X_{k'})_\Lambda \ar[r]^-{\lambda_{X_{k'}}} & H^{2d}(X_{k'}, \Lambda(d)),}
\end{equation}
where $v \colon X_{k'} \to X$ is the base change map. The left vertical arrow exists
and is injective by \cite[Prop.~6.1]{Binda-Krishna}. Since $k'$ is infinite, the
bottom horizontal arrow is injective. It follows that $\lambda_X$ must be injective too.
This concludes the proof of Theorem \ref{thm:LW-MC}.
\qed

\section{A question of Barbieri-Viale and Kahn}\label{sec:BVKahn}
Let $k$ be an algebraically closed field of  characteristic zero and let $X$ be a
projective and reduced $k$-scheme of pure dimension $d$.
We shall now prove our application of the existence of the map
$\lambda_X$ given by Theorem~\ref{thm:LW-MC}.

In \cite[13.7.6]{BVKahn}, the authors refer that in a private correspondence,  Marc Levine outlined the construction of a cycle map $c\ell$ from $\CH_0^{LW}(X)$ to $H^{2d}(X, \Z(d))$  inducing, in particular, a morphism
\[c\ell_{\rm tors} \colon  \CH_0^{LW}(X)_{\rm tors} \longrightarrow H^{2d}(X,\Z(d))_{\rm tors}\]
that they conjecture to satisfy a number of properties. We can now give a positive answer to their conjecture.

We shall verify the expectations of  Barbieri-Viale and Kahn
by working with the modified version $\CH^{\lci}_0(X)$ instead 
of  $\CH_0^{LW}(X)$, keeping in mind that the two Chow groups 
actually agree under the above assumption on $k$, by \cite[Thm.~3.17]{Binda-Krishna}.

First, let $J^d(X)$ be the universal regular semi-abelian variety quotient of
$\CH_0^{LW}(X)_{\deg 0}$, constructed in \cite{BiswasSrinivas}.  This is universal for regular
homomorphisms (see op. cit. for the definition of a regular homomorphism)
from $\CH_0^{LW}(X)$  to semi-abelian varieties. It was shown in
\cite[Prop.~9.7]{Binda-Krishna} that $J^d(X)$ is also the universal regular
semi-abelian variety quotient of $\CH_0^{\lci}(X)_{\deg 0}$.

Next, let $\mathbf{L}_1\Alb^*(X)$ be the $1$-motive
\[ \mathbf{L}_1\Alb^*(X) = H_1^t(\mathrm{L}\Alb (M(X)^*(d)[2d])),\]
where $M(X)^*$ is the dual of $M(X)$ in $\mathbf{DM}(k)$, the homology $H_1^{t}(-)$ denotes the $H_1(-)$ homology with respect to the $t$-structure (introduced in \cite[3.1]{BVKahn}) on Deligne's category of $1$-motives  $D^b(\mathcal{M}_1)$, and finally $\mathrm{LAlb}(-)$ denotes the integrally defined derived Albanese functor
\[\mathrm{LAlb}\colon \mathbf{DM}_{\rm gm}^{\rm eff}(k) \longrightarrow D^b(\mathcal{M}_1), \] 
introduced in \cite[Def.~2.1.1]{BVKahn} (note that $M(X)^*(d)[2d]$ is effective, so that the definition makes sense, and that we are working in characteristic zero). In particular, $\mathbf{L}_1\Alb^*(X)$ is a semi-abelian variety.
By \cite[(13.7.1)]{BVKahn}, there is a canonical map
\begin{equation}\label{eq:u1}u\colon H^{2d}(X, \Z(d)) \longrightarrow \mathbf{L}_1\Alb^*(X)(k)\end{equation}
that is an isomorphism on the torsion subgroups by \cite[Corollary 13.7.4]{BVKahn}.

\vskip .3cm

We now have all the ingredients to state and prove the following result.
This verifies all expectations of Barbieri-Viale and Kahn.

\begin{thm} Let $X$ and $k$ be as above. Then the morphism
  \[\lambda_X\colon \CH_0^{\lci}(X) \to H^{2d}(X, \Z(d))\] is surjective with uniquely divisible kernel, and  there is a commutative diagram
  \begin{equation}\label{eqn:BKahn}
\begin{tikzcd}
  \CH_0^{\lci}(X)_{\rm tors} \arrow[r, "\lambda_X"] \arrow[d, "a^+"]& H^{2d}(X, \Z(d))_{\rm tors} \arrow[d, "u"] \\
   \Alb^+(X)(k)_{\rm tors} \arrow[r] & \mathbf{L}_1\Alb^*(X)(k)_{\rm tors}
    \end{tikzcd}
\end{equation}
    where all the arrows are isomorphisms. 
\end{thm}
\begin{proof}
  The existence and explicit construction of $\lambda_X$ was shown
  in Theorem~\ref{thm:LW-MC}. To check that \eqref{eqn:BKahn} commutes, it
  suffices to check it for the cycle class of a closed point $x \in X_\reg$.
  This reduces to checking the commutativity for points where this is well known.

 Now, the left vertical arrow  in \eqref{eqn:BKahn} is an isomorphism by the main result of \cite{BiswasSrinivas}.
The right vertical arrow is an isomorphism by \cite[13.7.5]{BVKahn}.
The bottom horizontal arrow is an isomorphism by
\cite[Thm.~12.12.6]{BVKahn}. 
Thus every arrow in \eqref{eqn:BKahn} is an isomorphism.
Finally, recall (see, e.g., \cite[Lemma 5.1]{BK3}) that since $k$ is algebraically closed, the subgroup $\CH_0^{\lci}(X)_{\deg 0}$ is divisible. Since $\lambda_X$ is an isomorphism on torsion by the above discussion, an easy diagram chase implies that the kernel of $\lambda_X$ is uniquely divisible, completing the proof of the theorem.
\end{proof}

\end{document}